\pdfoutput=1
\documentclass[reqno, 10pt]{amsart}
\usepackage{amsmath,amssymb,mathrsfs,amsthm,amsfonts,mathtools}
\usepackage[inline]{enumitem} 				
\usepackage[usenames,dvipsnames]{xcolor}
\definecolor{hypercolor}{HTML}{003399}
\usepackage{hyperref}
\hypersetup{%
	colorlinks=true, linkcolor=hypercolor,
	citecolor=hypercolor
}
\usepackage{newtxtext}						
\usepackage{graphicx}
\usepackage[paper=letterpaper,margin=1in]{geometry}

\newcommand{\rev}{\textcolor{black}}

\newtheorem{thm}{Theorem}[section]
\newtheorem{lem}[thm]{Lemma}
\newtheorem{prop}[thm]{Proposition}
\newtheorem{cor}[thm]{Corollary}
\theoremstyle{definition}

\newtheorem{rmk}[thm]{Remark}

\allowdisplaybreaks
\newcommand*{\Cdot}{{\raisebox{-0.5ex}{\scalebox{1.8}{$\cdot$}}}} 
\usepackage{acronym}
\acrodef{LDP}{Large Deviation Principle}
\acrodef{KPZ}{Kardar--Parisi--Zhang}
\acrodef{SHE}{Stochastic Heat Equation}
\acrodef{WNT}{Weak Noise Theory}
\acrodef{LPP}{Last Passage Percolation}

\newcommand{\tcalK}{\widetilde{\calK}}
\newcommand{\calK}{\mathcal{K}}

\newcommand{\e}{\varepsilon}
\newcommand{\z}{\zeta}                     
\newcommand{\GG}{G}
\newcommand{\Leb}{\text{Leb}}               
\newcommand{\infi}{q_\scl}                  

\newcommand{\R}{\mathbb{R}}
\newcommand{\Z}{\mathbb{Z}}

\newcommand{\Csp}{C}						
\newcommand{\Cc}{C_c^\infty}                       
\newcommand{\Lsp}{L}						
\newcommand{\bana}{\mathcal{B}}				
\newcommand{\SD}{\mathsf{SD}}				
\newcommand{\Dom}{\mathcal{D}}
\renewcommand{\gg}{\mathsf{G}}
\newcommand{\sech}{\operatorname{sech}}
\newcommand{\pot}{\varphi}                     
\newcommand{\dev}{\rho}						
\newcommand{\rr}{\mathsf{r}}
\newcommand{\tdev}{\widetilde{\dev}}
\newcommand{\devm}{\rho_*}					
\newcommand{\devs}{\rho^{\mathsf{s}}}

\newcommand{\scl}{\lambda}					
\newcommand{\bb}{B_\text{b}}				
\newcommand{\hk}{p}

\newcommand{\ZZ}{Z}							
\newcommand{\ZZe}{\ZZ_\e} 

\newcommand{\Zfn}{\mathsf{Z}}				
\newcommand{\Zfnn}{\widetilde{\mathsf{Z}}}  

\newcommand{\hfn}{\mathsf{h}}				
\newcommand{\hfnn}{\widetilde{\hfn}}        
\newcommand{\rateone}{\Phi}					

\newcommand{\rate}{I}

\renewcommand{\P}{\mathbb{P}}				
\newcommand{\E}{\mathbb{E}}					
\renewcommand{\d}{\mathrm{d}}				
\newcommand{\ind}{\mathbf{1}}				

\newcommand{\norm}[1]{\Vert #1\Vert}		
\newcommand{\normL}[1]{\| #1\|_{L^2(\R)}}	
\newcommand{\normLL}[1]{\|#1\|_{2}}

\newcommand{\dist}{\mathrm{dist}}

\newcommand{\indentation}{\noindent}

\newcommand\numberthis{\addtocounter{equation}{1}\tag{\theequation}}
\numberwithin{equation}{section}
\begin{document}
\title{KPZ equation with a small noise, deep upper tail and limit shape}
\author{Pierre Yves Gaudreau Lamarre, Yier Lin, and Li-Cheng Tsai}
\address[Pierre Yves Gaudreau Lamarre]{\ Department of Statistics, University of Chicago}
\email{pyjgl@uchicago.edu}
\address[Yier Lin]{\ \hspace{72pt}Department of Statistics, University of Chicago}
\email{ylin10@uchicago.edu}
\address[Li-Cheng Tsai]{\ \hspace{53pt}Departments of Mathematics, University of Utah}
\email{lctsai.math@gmail.com}
\begin{abstract}
In this paper, we consider the KPZ equation under the weak noise scaling. That is,  we introduce a small parameter $\sqrt{\e}$ in front of the noise and let $\e \to 0$. 
We prove that the one-point large deviation rate function has a $\frac{3}{2}$ power law in the deep upper tail. Furthermore, by forcing the value of the KPZ equation at a point to be very large, we prove a limit shape of the solution of the KPZ equation as $\e \to 0$. This confirms the physics prediction in \cite{kolokolov2007optimal, kolokolov2009explicit,  kamenev2016short, meerson2016large, le2016exact, hartmann2019optimal}. 
\end{abstract}
\maketitle
\section{Introduction}
\indentation 
The \ac{KPZ} equation \cite{kardar1986dynamic} is a non-linear stochastic PDE which describes the random growth of an interface that has a property of lateral growth and relaxation  
\begin{equation}\label{eq:kpzunscaled}
\partial_t h = \frac{1}{2} \partial_{xx} h + \frac{1}{2} (\partial_x h)^2 + \xi.
\end{equation}
Here $\xi$ is the space-time white noise, which can be informally understood as a Gaussian field with Dirac-delta correlation function $\E[\xi(t, x) \xi(s, y)] = \delta(t-s) \delta(x-y)$. The KPZ equation has been studied intensively over the past 35 years. We refer to \cite{ferrari2010random, quastel2011introduction, corwin2012kardar, quastel2015one, chandra2017stochastic,corwin2020some} for some surveys of the mathematical studies of the \ac{KPZ} equation.
\bigskip
\\
Care is needed to make sense of the solution to \eqref{eq:kpzunscaled} due to the non-linearity and space-time white noise in the equation. One way of defining the solution is through the \emph{Hopf-Cole transform}. That is, we define $h := \log Z$, where $Z$ solves the \ac{SHE}
\begin{equation*}
\partial_t Z = \frac{1}{2} \partial_{xx} Z + \xi Z.
\end{equation*} 
We say that $Z$ is the \emph{mild solution} to the SHE if
\begin{equation}\label{eq:mild}
Z(t, x) = \int_\R p(t, x- y) Z(0, y) dy + \int_0^t \int_\R p(t-s, x-y) Z(s, y) \xi(s, y) dsdy,
\end{equation}
where $p(t, x) := \frac{1}{\sqrt{2\pi t}} e^{-\frac{x^2}{2t}}$ is the heat kernel. The solution theory of the \ac{SHE} is standard;
see \cite[Sections 2.1-2.6]{quastel2011introduction} for more details. Moreover, for function-valued initial data $Z(0, \Cdot) \geq 0$ that is not identically zero, \cite{mueller1991support} shows that $Z$ is always positive, i.e. almost surely $Z(t, x) > 0$ for all $t > 0$ and $x \in \R$. This guarantees the wellposedness of $h$. One often considered initial data is $Z(0, \Cdot) = \delta(\Cdot)$, where $\delta(\Cdot)$ is a Dirac-delta function. We refer to this as the Dirac-delta initial data for $Z$ and the \emph{narrow wedge initial data} for $h$. \cite{flores2014strict} shows that under the Dirac-delta initial data, almost surely $Z$ is positive for all $t > 0$ and $x \in \R$. 
Other definitions and constructions of the solution to the \ac{KPZ} equation are given by regularity structure \cite{hairer2014theory}, paracontrolled distribution \cite{gubinelli2015paracontrolled} or the notion of energy solution \cite{gonccalves2014nonlinear, gubinelli2018energy}.
\bigskip
\\
In recent years, the large deviations of the \ac{KPZ} equation have received much attention in the mathematics and physics communities. The large deviations of the \ac{KPZ} equation can be studied in two regimes: long time regime $(t \to \infty)$  and short time regime $(t \to 0)$. For the long time regime, 
the work \cite{corwin2020lower}  rigorously proved a detailed bound for the lower tail of the KPZ equation under the narrow wedge initial data. This bound captures a cubic to $\frac{5}{2}$ crossover; see also the physics work \cite{krajenbrink2018simple}. \cite{corwin2020kpz} obtained similar bounds for the \ac{KPZ} equation under general initial data. Under the narrow wedge initial data, the exact one-point lower tail large deviation rate function was derived in the physics works \cite{sasorov2017large, corwin2018coulomb, krajenbrink2018systematic, doussal2019large} and was proved rigorously by \cite{tsai2018exact, cafasso2021riemann}. \cite{krajenbrink2019linear} showed that the four methods in \cite{sasorov2017large, corwin2018coulomb, krajenbrink2018systematic, tsai2018exact} are closely related. For the upper tail, the physics work \cite{le2016large} predicted the $\frac{3}{2}$-power law for the entire rate function of the KPZ equation narrow wedge initial data. \cite{das2021fractional} gave a rigorous proof for the upper tail \ac{LDP}. The result was extended to general initial data by \cite{ghosal2020lyapunov}.  
\bigskip
\\
For the large deviations of the \ac{KPZ} equation in the short time regime, the results are fruitful in the physics literature; see \cite{krajenbrink:tel-02537219}. In particular, the physics literature 
\cite{kolokolov2007optimal, kolokolov2009explicit, kamenev2016short, meerson2016large} predicted that for the narrow wedge and flat initial data, the one-point large deviation rate function exhibits a $\frac{3}{2}$-power law in the deep upper tail, a quadratic power law in the near-center tail and a $\frac{5}{2}$-power law in the deep lower tail. 
The physics work \cite{le2016exact} derived the entire one-point rate function, from which the authors are able to confirm these power laws. Their prediction was backed by the numerical result \cite{hartmann2018high}. The one-point large deviations 
were rigorously proved in \cite{lin2021short}. 
The authors also rigorously proved the quadratic to $\frac{5}{2}$-power law crossover in the lower tail rate function. 
\bigskip
\\
Studying the KPZ equation in the short time regime is the same as studying the KPZ equation in the \emph{weak noise regime}. That is,
we introduce a small parameter $\sqrt{\e}$ in front of the noise,
\begin{equation}\label{eq:kpzweaknoise}
\partial_t h_\e  = \tfrac{1}{2} \partial_{xx} h_\e + \tfrac{1}{2} (\partial_x h_\e)^2 + \sqrt{\e} \xi.
\end{equation}
The solution to the above equation is defined to be $h_\e := \log Z_\e$ where $Z_\e$ solves the \ac{SHE}
\begin{align}\label{e.mild.nw}
\partial_t Z_\e = \tfrac12 \partial_{xx} Z_\e + \sqrt{\e} \xi Z_\e.
\end{align}
Throughout the paper, 
we set $Z_\e (0, \Cdot) = \delta(\Cdot)$. The short time regime of \eqref{eq:kpzunscaled} is related to \eqref{eq:kpzweaknoise} through scaling, namely, $h(\e^2 \Cdot, \e \Cdot) + \log \e  \overset{d}{=} h_\e (\Cdot, \Cdot)$. We add $\log \e$ to guarantee that $h_\e$ starts from the narrow wedge initial data. 
\bigskip
\\ 
The \ac{LDP} of $Z_\e = Z_\e (\Cdot, \Cdot)$ under the limit $\e \to 0$ has been rigorously proven in \cite{lin2021short}. The rate function is of Freidlin-Wentzell type. In particular, by the contraction principle,  for $\scl \geq 0$,
\begin{align}
\label{eq:ratelower}
\lim_{\e \to 0} \e \log \P\big[h_\e (2, 0) + \log \sqrt{4\pi} \leq -\scl\big] &= -\rateone(-\scl),\\
\label{eq:rateupper}
\lim_{\e \to 0} \e \log \P\big[h_\e (2, 0) + \log \sqrt{4\pi} \geq \scl\big] &= -\rateone(\scl),
\end{align} 
where $\Phi$ is the infimum of the Freidlin-Wentzell rate function subject to the relevant constraint. 
Extracting the asymptotics of 
$\Phi$ 
is non-trivial. \cite{lin2021short} proved that $\lim_{\scl \to 0} \scl^{-2} \rateone(\scl) = \tfrac{1}{\sqrt{2\pi}}$ and $\lim_{\scl \to \infty}  \scl^{-\frac{5}{2}} \Phi(-\scl) = \tfrac{4}{15\pi}$.
\subsection{Main results.}
The first result of the current paper concerns the deep upper tail of the rate function $\Phi$. In other words, we look at the asymptotic of $\Phi(\scl)$ as $\scl \to \infty$. 
It has been predicted in the physics literature \cite{kolokolov2007optimal, kolokolov2009explicit, kamenev2016short, meerson2016large, le2016exact} that 
$
\lim_{\scl\to\infty} \scl^{-3/2} \rateone(\scl) = \tfrac43$. Our first result gives a rigorous proof of the $\frac{3}{2}$-power law in the deep upper tail.
\begin{thm}\label{thm:uptail}
We have $ \displaystyle\lim_{\scl\to\infty} \scl^{-3/2} \rateone(\scl) = \tfrac43 $.
\end{thm}
\indentation 
The second result of the current paper proves the limit shape of (the solution of) the \ac{KPZ} equation under the  weak noise scaling and the deep upper tail conditioning. This limit shape was predicted in the physics works \cite{kolokolov2007optimal, kolokolov2009explicit, kamenev2016short, meerson2016large, hartmann2019optimal}.
\begin{thm}\label{thm:limshape}
Define $h_{\e, \lambda} = \lambda^{-1} h_\e (t, \lambda^{\frac{1}{2}} x)$. For arbitrary fixed $\delta > 0$, we have
\begin{equation*}
\lim_{\lambda \to \infty}\lim_{\e \to 0}\mathbb{P}\big[\norm{h_{\e, \lambda} - \hfn_*}_{L^\infty ([\delta, 2] \times [-\delta^{-1}, \delta^{-1}])} < \delta\, \big|\, h_\e (2, 0) + \log \sqrt{4\pi} \geq \scl \big] = 1.
\end{equation*}
Here, we use $\|\Cdot\|_{L^\infty(D)}$ to denote the $L^\infty$ norm on the domain $D$. The limit shape $\hfn_*$ is given by
\begin{equation}
\label{eq:h*}
\hfn_*(t, x) := 
\begin{cases}
-|x| + \frac{t}{2}, &\text{when } |x| \leq t, \\
-\frac{x^2}{2t}, &\text{when } |x| \geq t.
\end{cases}
\end{equation}
See Figure \ref{fig:shape} for illustration.
\end{thm}
\begin{figure}[ht]
\centering
\includegraphics[scale = 0.6]{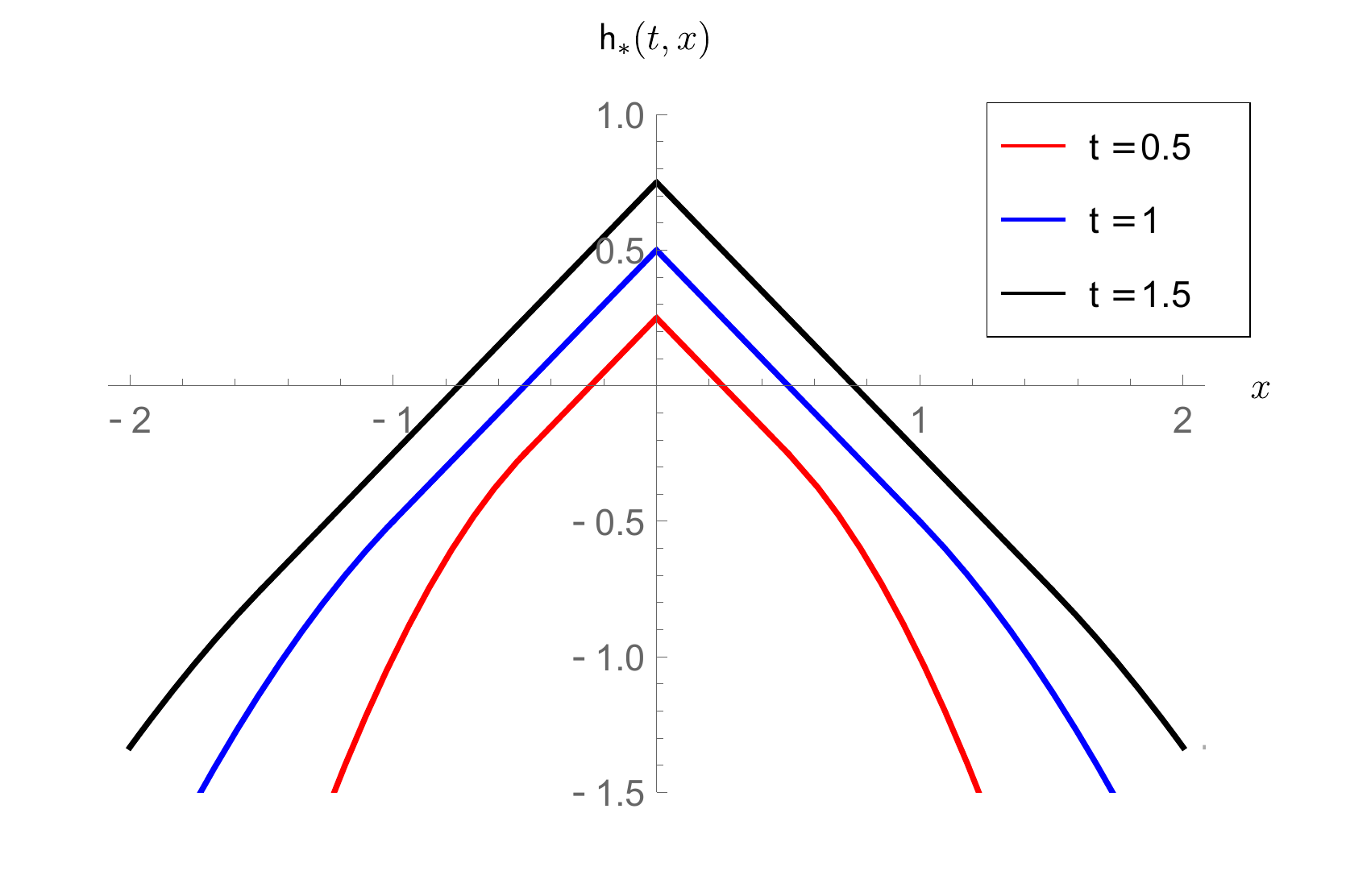}
\caption{The graph of $\hfn_* (t, \Cdot)$, when $t = 0.5, 1, 1.5$.}
\label{fig:shape}
\end{figure}

\begin{rmk}\rev{%
Theorem~\ref{thm:uptail} gives the $ \lambda\to\infty $ limit of $ h_{\e,\lambda} $ under the upper-tail conditioning.
A natural related question is to obtain the limit under the lower-tail conditioning $  h_\e (2, 0) + \log \sqrt{4\pi} \leq -\scl $.
The latter question has recently been solved in \cite{lin22}.
We emphasize that the mechanisms for the large deviations are very different in the upper- and lower-tail conditioning.
In the upper-tail conditioning, the contribution of the noise $ \xi $ concentrates around $ x=0 $; in the lower-tail conditioning, the contribution of the noise spans a wide region in spacetime.
The required analysis in the current paper and in \cite{lin22} hence differ.
}\end{rmk}

\begin{rmk}
The limit shape $\hfn_*$ was predicted earlier in the physics work \cite{kolokolov2007optimal, kolokolov2009explicit, kamenev2016short, meerson2016large} via the weak noise theory, and in \cite{hartmann2019optimal} by simulations. Recently, the physics work \cite{krajenbrink2021inverse} solved the finite $\scl$ limit shape of $h_{\e, \scl}$. Using this, they confirmed the rate function $\Phi$ discovered in \cite{le2016exact}. In the limit of large $\scl$, they discussed how the form \eqref{eq:h*} emerges in the exact solution. %
\end{rmk}
\begin{rmk}
The reason that we set $h_{\e} (2, 0) + \log \sqrt{4\pi} > \scl$ in the theorem instead of $h_\e (2, 0) > \scl$ is purely for the convenience of the proof. It makes no difference since we let $\scl \to \infty$.
\end{rmk}

\begin{rmk}
\rev{Our method does not rely on exact formulas and may apply to other initial data.
In particular, our method should apply to the flat initial data $h_\e(0, \Cdot) = 0$. 
The result in Theorem \ref{thm:uptail} remains the same for the flat initial data, while Theorem \ref{thm:limshape} holds with a different limit shape 
\begin{equation*}
\hfn^{\text{flat}}_* (t, x) := 
\begin{cases}
-|x| + \frac{t}{2}, &\text{when } |x| \leq \frac{t}{2}, \\
0, &\text{when } |x| > \frac{t}{2}.
\end{cases}
\end{equation*} 
More broadly, one can consider applying our method to a function-valued, symmetrically decreasing initial data: $ h_\text{ic}(x) = h_\text{ic}(|x|) $ and $ h_\text{ic}(|x|) $ non-increasing in $ |x| $.
We conjecture that the result in Theorem \ref{thm:uptail} remains the same; the limit shape (in-general) needs to be adjusted according to the initial data.}

\rev{Going beyond symmetrically decreasing initial data, one may see different behaviors of the deviations.
In particular, a dynamical phase transition triggered by a symmetry breaking has been predicted in \cite{janas16,smith18} (see also \cite{krajenbrink17short,hartmann2021observing,krajenbrink21flat}) for the Brownian initial data.
For such initial data, we do not expect our method to apply directly and new ideas are needed.}
\end{rmk}

\subsection{A review of the Freidlin-Wentzell \ac{LDP} for the \ac{SHE}}
\indentation 
To motivate the proof of Theorems \ref{thm:uptail} and \ref{thm:limshape}, we recall the Freidlin-Wentzell \ac{LDP} for the \ac{SHE} $\{Z_\e\}_{\e > 0}$. The result was established for the \ac{SHE} under function-valued initial data and the narrow wedge initial data in \cite[Proposition 1.7]{lin2021short}. 
For our propose, we only state the result for the narrow wedge initial data. 
\bigskip
\\
Let us first recall the definition of an \ac{LDP}. Let $\Omega$ be a topological space. We say that a sequence of $\Omega$-valued random variables $\{Y_\e\}_{\e > 0}$ \textbf{satisfies an \ac{LDP} with speed $\e^{-1}$ and rate function $I$} if  
\begin{align*}
\limsup_{\e \to 0} \e \log\P(Y_\e \in F) &\leq -\inf_{x \in F} I(x) \qquad \text{ if } F \subseteq \Omega \text{ is  closed},\\
\liminf_{\e \to 0} \e \log\P(Y_\e \in G) &\geq -\inf_{x \in G} I(x) \qquad  \text{ if } G \subseteq \Omega \text{ is  open}.
\end{align*}
\indentation 
We 
now state the \ac{LDP} for the SHE $\{Z_\e\}_{\e > 0}$. 
Fix $T > 0$ and $\delta \in (0, T)$. We take $\Omega = C([\delta, T] \times 
[-\delta^{-1}, \delta^{-1}])$ with the uniform topology and view $Z_\e$ as an $\Omega$-valued random variable. The reason that we 
avoid $t = 0$ in our choice of $\Omega$ is because $Z_\e$ starts from the Dirac-delta initial data, which is singular. For $t > 0$, the heat kernel in \eqref{eq:mild} smoothes out the singularity, so $Z_\e$ is $\Omega$-valued for any fixed $\delta > 0$; see \cite{quastel2011introduction}.  
\bigskip
\\
Since $Z_\e$ in \eqref{e.mild.nw} is driven by the space-time white noise $\sqrt{\e} \xi$, it would be helpful to first look at the \ac{LDP} of $\{\sqrt{\e} \xi\}_{\e > 0}$. We view $\dev \in L^2([0, T] \times \R)$ as a deviation $\sqrt{\e} \xi$. Since $\xi$ has Dirac-delta correlation function, 
we have informally $\P(\sqrt{\e}\xi \approx \dev) \approx \exp(-\frac{1}{2} \e^{-1}\|\dev\|^2_{L^2([0, T] \times \R)})$ for small $\e$. We replace the noise $\sqrt{\e} \xi$ with its deviation $\dev$ and consider the PDE
\begin{equation}\label{eq:pdeintro}
\partial_t \Zfn = \frac{1}{2} \partial_{xx} \Zfn + \dev \Zfn,\qquad \Zfn(0, x) = \delta(x),
\end{equation} 
where $\Zfn = \Zfn(\dev; t, x)$, $t \in [0, T]$ and $x \in\R$. Like the \ac{SHE}, the solution to this PDE is understood in the mild form
\begin{equation}\label{eq:pdemild}
\Zfn(\dev; t, x) = p(t, x) + \int_0^t \int_\R p(t-s, x-y) \dev(s, y) \Zfn(\dev; s, y) dsdy.
\end{equation}
By iteration of \eqref{eq:pdemild}, the solution $\Zfn(\rho)$ admits a series expansion and we have the Feynman-Kac formula
\begin{equation}\label{eq:feynmankac}
\Zfn(\rho; t, x) = \E_{0 \to x}\bigg[\exp\Big(\int_0^t \rho(s, \bb(s)) ds\Big)\bigg] p(t, x),\qquad t \in (0, T] \times \R, 
\end{equation}
where $\bb$ is a Brownian bridge such that $\bb(0) = 0$ and $\bb(t) = x$. 
\bigskip
\\
It is standard to see that $\Zfn(\dev)$ is a continuous function on $(0, T] \times \R$, we refer to \cite[Section 2]{lin2021short} for more detail. Fix $\delta > 0$. We view $\Zfn: \dev \mapsto \Zfn(\dev)$ as a map from $L^2([0, T] \times \R)$ to $C([\delta, T] \times [-\delta^{-1}, \delta^{-1}])$. We now state the \ac{LDP} for $\{Z_\e\}_{\e > 0}$. 
\begin{prop}\label{prop:FWLDP}
	Fix $ T<\infty$ and $0 < \delta < T$. 
	Let $ \ZZe $ be the solution to \eqref{e.mild.nw} with the Dirac-delta initial data. 
	$ \{\ZZ_\e\}_{\e > 0} $ satisfies an \ac{LDP} in $ \Csp([\delta, T] \times [-\delta^{-1}, \delta^{-1}]) $ with speed $ \e^{-1} $ and the rate function 
	\begin{equation*}
	\label{e.rate}
	\rate(f) := \inf\big\{\tfrac{1}{2}\|\dev\|_{L^2([0, T] \times \R)}^2 \,:\, \dev \in \Lsp^2([0, T] \times \R),  \Zfn(\dev) = f \big\},\quad f \in C([\delta, T] \times [-\delta^{-1}, \delta^{-1}]),
	\end{equation*}
where $\Zfn(\dev)$ is the unique solution to 
$ \partial_t \Zfn = \frac12 \partial_{xx} \Zfn + \dev \Zfn $ with the Dirac-delta initial data $ \Zfn(0,x) = \delta(x) $.
\end{prop}
\begin{proof}
This is a direct consequence of \cite[Proposition 1.7]{lin2021short} Part (b). 
\end{proof}
\indentation Since $Z_\e = e^{h_\e}$, we have $\P[h_\e (2, 0) + \log \sqrt{4\pi} \geq \scl] = \P[Z_\e (2, 0) \geq \frac{1}{\sqrt{4\pi}} e^{\scl}]$. By \eqref{eq:rateupper}, we have $\Phi(\scl) = -\lim_{\e \to 0} \e \log \P[Z_\e (2, 0) \geq \frac{1}{\sqrt{4\pi}} e^\scl]$.
Taking  $T = 2$ in Proposition \ref{prop:FWLDP} and applying contraction principle yield that for $\scl \geq 0$, 
\begin{equation}
\rateone(\scl) =  
\label{eq:ptrate}
\inf\Big\{ \frac12 \|\dev\|_{L^2([0, 2] \times \R)}^2 \, : \dev \in L^2([0, 2] \times \R), 
\, \Zfn(\dev;2,0) \geq \frac{1}{\sqrt{4\pi}} e^{\scl} \Big\}.
\end{equation}
\subsection{Proof ideas}
\label{sec:idea}
We explain the ideas for 
proving Theorems \ref{thm:uptail} and \ref{thm:limshape}. 
The first step towards proving Theorem \ref{thm:uptail} is to apply a scaling to the variational formula \eqref{eq:ptrate}. By the Feynman-Kac formula \eqref{eq:feynmankac}, we know that for arbitrary $\dev \in L^2([0, 2] \times \R)$,  
$\Zfn(\scl\dev(\scl \Cdot, \scl^{\frac{1}{2}} \Cdot); t, x) = \scl^{\frac{1}{2}} \Zfn(\dev; \scl t; \scl^{\frac{1}{2}} x).$ 
Using this relation and applying the scaling $\dev \mapsto \scl \dev(\scl \Cdot, \scl^{\frac{1}{2}} \Cdot)$, one can rewrite the rate function in \eqref{eq:ptrate} as
\begin{equation}\label{eq:scaledinf}
\Phi(\scl) = \scl^{\frac{3}{2}} \inf \Big\{\frac{1}{2\lambda}\|\dev\|_{L^2([0, 2\scl] \times \R)}^2:\dev \in L^2([0, 2\scl] \times \R),\, \Zfn(\dev; 2\scl, 0) \geq \frac{1}{\sqrt{4\pi \scl}} e^\scl\Big\}.
\end{equation} 
The value of the infimum should be of constant order, this explains the $\tfrac{3}{2}$-power in Theorem \ref{thm:uptail}. 
To prove Theorem \ref{thm:uptail}, it suffices to show that 
\begin{equation}\label{eq:heuristic1}
\lim_{\scl \to \infty}\inf \Big\{\frac{1}{2\lambda}\|\dev\|_{L^2([0, 2\scl] \times \R)}^2: \dev \in L^2([0, 2\scl] \times \R),\, \Zfn(\dev; 2\scl, 0) \geq \frac{1}{\sqrt{4\pi \scl}} e^\scl\Big\} = \frac{4}{3}.
\end{equation}
We explain why the $\frac{4}{3}$ appears on the right hand side. To motivate the discussion, 
let us assume that the minimizers (there might be more than one) of \eqref{eq:heuristic1} become asymptotically \emph{time-independent} as $\scl \to \infty$. 
Under this assumption, 
we only need to consider $\rho$ such that $\dev(t, \Cdot) = \varphi(\Cdot)$ for $t \in [0, 2\scl]$ and some $\varphi:\R \to \R$. So \eqref{eq:heuristic1} simplifies to 
\begin{equation}\label{eq:heuristic2}
\lim_{\scl \to \infty} \inf\Big\{ \|\varphi\|_{L^2(\R)}^2: \varphi \in L^2(\R),\, \Zfn(\varphi; 2\scl, 0) \geq \frac{1}{\sqrt{4\pi \scl}} e^{\scl}\Big\}  = \frac{4}{3}.
\end{equation}  
Recall that $\Zfn(\varphi)$ solves the PDE $\partial_t \Zfn = \frac12 \partial_{xx} \Zfn + \varphi \Zfn$. Since $\varphi$ does not depend on time, we treat $\Zfn(t)$ as function-valued and view the PDE as a function-valued ODE  
$\partial_t \Zfn(t) = A^\varphi \Zfn(t)$, where $A^\varphi := \frac{1}{2} \partial_{xx} + \varphi$ is a Schr\"{o}dinger-type operator.  
Set the largest eigenvalue of $A^\varphi$ 
to be $F(\varphi)$, whose expression is given in \eqref{eq:funcF}. It is natural to expect that $\Zfn(\rho; t, 0)$ at large time $t$ grows as $\exp(tF(\varphi))$.
In particular, we have $\Zfn(\varphi; 2\scl, 0) \sim \exp(2\scl F(\varphi))$ as $\scl \to \infty$. This suggests that 
\begin{equation}\label{eq:heuristic3}
\text{LHS of } \eqref{eq:heuristic2}  = \inf\Big\{\|\varphi\|^2_{L^2(\R)}: F(\varphi) \geq \frac{1}{2}\Big\}.
\end{equation}  
The problem of finding the minimizers of \eqref{eq:heuristic3} can be transformed into understanding when the equality of the Gagliardo–Nirenberg-Sobolev inequality holds, see Lemma \ref{lem:potbd}. It turns out that the minimizers of the infimum are given by 
$\sech^2(x + v)$, where $v$ is an arbitrary constant. Since $\|\sech^2(x+v)\|_{L^2(\R)}^2 = \frac{4}{3}$ for all $v$, we know that the right hand side of \eqref{eq:heuristic3} equals $\tfrac{4}{3}$. This explains why Theorem \ref{thm:uptail} should hold.
\bigskip
\\
We proceed to explain the idea for proving Theorem \ref{thm:limshape}. Under the time-independence assumption, it is natural to believe that the minimizers of the infimum in \eqref{eq:scaledinf} converges to the minimizers of the infimum in \eqref{eq:heuristic3} as $\scl \to \infty$. However, from the previous paragraph, the minimizer of the infimum in \eqref{eq:heuristic3} is not unique and this causes \rev{a} problem to our analysis. To resolve this, we show that any minimizer of the infimum in \eqref{eq:scaledinf} has to be symmetric in the $x$ coordinate (in fact the function should reach its peak at $0$ and decrease on both sides of $0$). The proof of this is carried out in Section \ref{sec:sdf}. Consequently, we see that the minimizer of \eqref{eq:heuristic1} should converge to $\sech^2 x$ as $\scl \to \infty$.
\bigskip
\\
We denote $\devm(t, x) := \devm (x) := \sech^2 x$. By the Freidlin-Wentzell \ac{LDP} stated in Proposition \ref{prop:FWLDP} and scaling, one can show that for fixed $\scl$,  as $\e \to 0$, the space-time path $h_{\e, \scl}$ concentrates around $\scl^{-1}\log \Zfn(\devm; \scl \Cdot, \scl \Cdot)$. 
By the Feynman-Kac formula \eqref{eq:feynmankac},  
\begin{equation}\label{eq:hfnlimintro}
\scl^{-1} \log \Zfn (\dev_*; \scl t, \scl x) = \scl^{-1} \log \E_{\scl x \to 0}\bigg[\exp\Big(\int_0^{\scl t} \devm(\bb(s)) ds\Big)\bigg] - \frac{x^2}{2t} - \scl^{-1} \log \sqrt{4\pi}.
\end{equation}
We seek to compute the $\scl \to \infty$ limit of the first term on the right hand side of \eqref{eq:hfnlimintro}. 
\bigskip
\\
Assume $x \neq 0$. The Brownian bridge starts from $\scl x$, which is far way from $0$. Since $\devm (x) = \sech^2 (x)$ decays exponentially as $|x| \to \infty$, the path of $\bb$ contributes little to the integral until it arrives near $0$. 
Fix $s \in [0, t]$, the probability of the event that the first time $\bb$ hits $0$ around time $\scl s$ is approximately $\exp(-\frac{\scl x^2 (t-s)}{2 s t})$. 
After arriving at $0$, 
we have
$\E_{0 \to 0}[\exp(\int_{\scl s}^{\scl t} \devm(\bb(r)) dr)] \approx \exp(\frac{\scl (t-s)}{2})$ for large $\scl$. 
Optimizing over $s \in [0, t]$, we expect that 
\begin{equation*}
\lim_{\scl \to \infty} \scl^{-1} \log \E_{\scl x \to 0}\bigg[\exp\Big(\int_0^{\scl t} \devm(\bb(s)) ds\Big)\bigg] = \sup_{s \in [0, t]}\Big(
-\frac{x^2 (t-s)}{2 s t} + \frac{t-s}{2}\Big) = \hfn_*(t, x) + \frac{x^2}{2t}. 
\end{equation*}
Using this together with \eqref{eq:hfnlimintro} shows that $\scl^{-1} \log\Zfn(\devm; \scl \Cdot, \scl \Cdot)$ converges to $\hfn_*$ defined in \eqref{eq:h*}.
\subsection{Technical Difficulties}
In this section, we emphasize some of the technical difficulties in our proof. Denote the set of minimizers of \eqref{eq:scaledinf} to be $\calK_\scl$. As mentioned in the previous section, one important step in our proof is to show that $\calK_\scl$ converges to $\rho_*$ as $\scl \to \infty$. We explain the proof ingredients of it in Sections \ref{sec:property} and \ref{sec:convergence}. In Section \ref{sec:equicontinuity}, we describe a technical issue for proving the limit shape.   
\subsubsection{Properties of $\calK_\scl$}\label{sec:property} 
To prove that the elements of $\calK_\scl$ converge to $\rho_*$ as $\scl \to \infty$, 
we first need to show that $\calK_\scl$ is not empty, i.e. the minimizer of the infimum in \eqref{eq:scaledinf} exists. To prove this, we need 1). the compactness of the level sets  $\{\rho: \|\rho\|_{L^2([0, 2\scl] \times \R)}^2 \leq r\}$; 2). the continuity of the map $\rho \mapsto \Zfn(\rho; 2\scl, 0)$. Unfortunately, the level sets $\{\rho: \|\rho\|_{L^2([0, 2\scl] \times \R)}^2 \leq r\}$ are not compact in $L^2([0, 2\scl] \times \R)$. To overcome this, we consider an abstract Wiener space $(\mathcal{B}, \mu)$ such that $L^2([0, 2\scl] \times \R) \subseteq \mathcal{B}$ is the Cameron-Martin space. This gives us the compactness of $\{\rho: \|\rho\|_{L^2([0, 2\scl] \times \R)}^2 \leq r\}$ in $\mathcal{B}$ and it preserves the continuity of the map $\rho \mapsto \Zfn(\rho; 2 \scl, 0)$.
\bigskip
\\
We also need to show that the elements of $\calK_\scl$ are symmetric and decreasing in space (see Section \ref{sec:sdf} for the precise definition). We prove this by considering the symmetric and decreasing rearrangement of $\rho$. Using the rearrangement inequalities from \cite{brascamp1974general, LiebLoss}, we show that the symmetric and decreasing rearrangement of $\rho$ preserves the $L^2$ norm while it increases the value of $\mathsf{Z}(\rho; 2\scl, 0)$. This implies that the elements in $\calK_\scl$ must be symmetric and decreasing in space.
\subsubsection{$\calK_\scl$ converges to $\rho_*$}\label{sec:convergence} 
The key innovations for proving the convergence are 1). A near-optimal upper bound of $\Zfn(\rho; 2\scl, 0)$ in terms of an exponential integral of the ground state $F(\rho)$, see Proposition \ref{p.l2tol2}; 2). A perturbative analysis of $F$, see Lemma \ref{lem:reversebd}. 
\bigskip
\\
For 1), one can obtain an upper bound of $\Zfn(\rho; 2\scl, 0)$ via iterating \eqref{eq:pdemild} and applying the Cauchy-Schwarz inequality, see Lemma \ref{l.ptL2}. However, such bounds are not optimal. In Proposition \ref{p.l2tol2}, we prove a near-optimal bound of $\Zfn(\rho; 2\scl, 0)$. The proof of the proposition relies on the time-dependent semigroup $\partial_t - \frac{1}{2} \partial_{xx} - \rho$, an energy estimate when $\rho$ is smooth and compactly supported, and the method of  approximation for $\rho \in L^2([0, 2\scl] \times \R)$, which is accomplished in Lemmas \ref{l.ptL2}, \ref{lem:continuity} and \ref{lem:Fcont}.  
\bigskip
\\
We explain 2) in more detail. Using the $L^4$ Gagaliardo-Nirenberg-Sobolev inequality, we obtain an optimal upper bound of $F(\varphi)$ as well as identify its optimizers; see Lemma \ref{lem:potbd}. The optimizer in Lemma \ref{lem:potbd} is only unique up to shifting. Our perturbative analysis says that when $\varphi$ is symmetric and decreasing and $F(\varphi)$ is near its optimal upper bound, $\varphi$ must be close to the symmetric optimizer. The proof of this is through a weak-convergence argument.
\bigskip
\\ 
The proof of the convergence is carried out in Proposition \ref{prop:devmdist}. Combining items 1) and 2), the condition $\rho \in \calK_\scl$ implies that for most $r \in [0, 2\scl]$, $\|\rho(r, \Cdot)\|_{L^2(\R)}$ is close to $\|\rho_*\|_{L^2(\R)}$ and $F(\rho(r, \Cdot))$ is close to $F(\rho_*)$, which is the optimal upper bound. This implies the convergence of $\calK_\scl$ to $\rho_*$, as $\scl \to \infty$.
\subsubsection{The equi-continuity}\label{sec:equicontinuity}
By the Freidlin-Wentzell LDP, we can show that for fixed $\scl > 0$, as $\e \to 0$, $h_{\e, \scl}$ concentrates around the set of functions $\scl^{-1} \log \scl^{\frac{1}{2}}\Zfn(\calK_\scl; \scl \Cdot, \scl \Cdot)$. In Section \ref{sec:convergence}, we explain how to prove that $\calK_\scl$ converges to $\rho_*$, as $\scl \to \infty$. To prove Theorem \ref{thm:limshape}, we need to show additionally that the distance between $\scl^{-1} \log \scl^{\frac{1}{2}} \Zfn(\calK_\scl; \scl \Cdot, \scl \Cdot)$ and $\scl^{-1} \log \scl^{\frac{1}{2}} \Zfn(\rho_*; \scl \Cdot, \scl \Cdot)$ is small as $\scl \to \infty$. This is proved by establishing the equi-continuity of the maps $f_\scl : \rho \mapsto  \scl^{-1} \log \scl^{\frac{1}{2}} \Zfn(\rho; \scl \Cdot, \scl \Cdot)$, see Proposition \ref{prop:uniformdist}.

\subsection*{Acknowledgments.}
\rev{We thank Ivan Corwin, Alexandre Krajenbrink, Pierre Le Doussal, and Baruch Meerson for their helpful comments on the presentation of this work.
We thank the referees for their useful comments on the manuscript, especially for pointing out an error in Lemma~\ref{lem:equicontinuity} in the first version of the manuscript.
The research of LCT is partially supported by the Sloan Fellowship and the NSF through DMS-1953407 and DMS-2153739.}

\subsection*{Outline of the rest of the paper} In Section \ref{sec:uptail}, we prove Theorem \ref{thm:uptail} and confirm the $\frac{3}{2}$-power law in the deep upper tail rate function. In Section \ref{sec:minimizer}, we give some detailed characterization of the minimizers in \eqref{eq:heuristic1}. We also prove that the asymptotic limit of these minimizers equals $\sech^2 x$. In Section \ref{sec:limitshape}, we establish some result about equi-continuity and prove the convergence of the left hand side of \eqref{eq:hfnlimintro} to $\hfn_*$, thus completing the proof of Theorem \ref{thm:limshape}.

\section{The $\frac{3}{2}$-power law}\label{sec:uptail}
\indentation By the discussion in Section \ref{sec:idea}, to prove Theorem \ref{thm:uptail}, it suffices to show \eqref{eq:heuristic1}.
In particular, \eqref{eq:heuristic1} follows if we can show  
\begin{align}\label{eq:limsuptail} \tag{$\mathsf{LimSup}$}
&\limsup_{\scl \to \infty}\inf \Big\{\frac{1}{2\scl}\|\dev\|_{L^2([0, 2\scl] \times \R)}^2: \dev \in L^2([0, 2\scl] \times \R),\, \Zfn(\dev; 2\scl, 0) \geq \frac{1}{\sqrt{4\pi \scl}} e^\scl\Big\} \leq \frac{4}{3},\\
\label{eq:liminftail} \tag{$\mathsf{LimInf}$}
&\liminf_{\scl \to \infty}\inf \Big\{\frac{1}{2\scl}\|\dev\|_{L^2([0, 2\scl] \times \R)}^2:\dev \in L^2([0, 2\scl] \times \R),\, \Zfn(\dev; 2\scl, 0) \geq \frac{1}{\sqrt{4\pi \scl}} e^\scl\Big\} \geq \frac{4}{3}.
\end{align}
The rest of the section is devoted to proving \eqref{eq:limsuptail} and \eqref{eq:liminftail}.
\bigskip
\\
For the rest of the paper, we will use $C = C(a_1, a_2,\dots)$ to denote a deterministic positive finite constant. The
constant may change from line to line or even within the same line, but depends only on
the designated variables $a_1, a_2,\dots$. In addition, we will denote the Brownian motion as $B$ and Brownian bridge as $\bb$. When we write $\E_{x \to y}[f(\int_a^b \bb(s) ds)]$, the expectation is taken with respect to a Brownian bridge with $\bb(a) = x$ and $\bb(b) = y$. When we write $\E_{x}[f(\int_a^b B(s) ds)]$, the expectation is taken with a Brownian motion starting from $B(a) = x$.
\subsection{Proof of \eqref{eq:limsuptail}}
The key to the proof is to consider the exponential moment of the $\bb$ subject to a time-independent potential $\varphi$. More precisely, for $\varphi \in L^2(\R)$, we consider  
\begin{equation}\label{eq:FK}
\Zfn(\varphi; 2\scl, 0) = \E_{0 \to 0}\bigg[\exp\Big(\int_0^{2\scl} \varphi(\bb(s)) ds\Big)\bigg] p(2\scl, 0).
\end{equation}	
To analyze \eqref{eq:FK},
let us recall some background knowledge from \cite[Section 4.1]{chen2010random}. Fixing a bounded continuous function $\pot: \R \to \R$, we define 
$T^\pot_t: L^2(\R) \to L^2(\R)$ as
\begin{equation*}
(T^\pot_t g)(x) = \E_{x} \bigg[\exp\Big(\int_0^t \pot(B(s)) ds\Big) g(B(t))\bigg].
\end{equation*}
By \cite[page 105]{chen2010random},  $(T^\pot_t)_{t \geq 0}$ forms a strongly-continuous, self-adjoint and non-negative semi-group on $L^2(\R)$.
Let $A^\pot$ be the generator of the semi-group $(T^\pot_t)_{t \geq 0}$. Denote the domain of the generator $A^\pot$ by $\Dom(A^\pot)$. Let $\Cc(\R)$ be the space of compactly supported smooth function. Let $H^1(\R)$ to be the Sobolev space $\{g \in L^2(\R): g' \in L^2(\R)\}$. By \cite[Theorem 4.1.2 and Lemma 4.1.3]{chen2010random}, $\Cc(\R) \subseteq \Dom(A^\pot) \subseteq H^1(\R)$. Further, for $g \in \Dom(A^\pot)$,
\begin{equation}\label{eq:gnt}
\langle A^\pot
 g, g\rangle_{L^2(\R)} = \int_\R \pot(x) g(x)^2 - \frac{1}{2} g'(x)^2 dx.
\end{equation} 
For $\pot \in L^2 (\mathbb{R})$, define 
\begin{equation}\label{eq:funcF}
F(\pot) := \sup\Big\{\int_\R \pot(x)g(x)^2 -\frac{1}{2} g'(x)^2 dx: g \in H^1(\mathbb{R}), \normL{g} = 1 \Big\}.
\end{equation}
Assume that $\pot$ is bounded and continuous. Since $\Cc(\R)$ is dense in $H^1(\R)$ (w.r.t. the $H^1$-norm), we have
\begin{equation}\label{eq:ffuncF}
F(\pot) = \sup\{\langle A^\pot g, g\rangle_{L^2(\R)}: g \in \Cc(\R), \normL{g} = 1\}.
\end{equation}
Further, we record two useful inequalities from \cite[Eq 4.1.25 and 4.1.29]{chen2010random}. For $g \in \Cc(\R)$,
\begin{equation}\label{eq:chenresult}
\exp\big(t \langle A^\pot g, g \rangle_{L^2(\R)} \big) \leq \langle T^\pot_t g, g \rangle_{L^2(\R)} \leq \exp\big(tF(\pot)\big).
\end{equation}
It is known that (see \cite[Thereom 4.1.6]{chen2010random}) for bounded continuous $\pot$, 
$$\lim_{\scl \to \infty} \rev{ \frac{1}{\scl} \log }  \E_{0}\bigg[\exp\Big(\int_{0}^{\scl} \pot(B(s)) ds\Big)\bigg] = F(\pot).$$
We prove a similar result for the Brownian bridge, allowing the starting position to deviate from zero by an amount less than $\mathcal{O}(\scl^{\frac{1}{2}})$.
\begin{lem}\label{lem:bboptime}
Fix bounded continuous $\pot: \R \to \R$. Fix $\alpha \in (0, 1/2)$. Uniformly for $|x| \leq \scl^\alpha$, 
	\begin{equation}\label{eq:bboptime1}
	\lim_{\scl \to \infty}  \frac{1}{\lambda} \log \E_{x \to 0}\bigg[\exp\Big(\int_0^{\scl} \pot(B_b(s)) ds\Big)\bigg] = F(\pot).
	\end{equation}
\end{lem}
\begin{rmk}
Here we only need the $x = 0$ result for proving \eqref{eq:limsuptail}, and the result for $x \neq 0$  will be needed in Section \ref{sec:limitshape}.
\end{rmk}
\indentation We say that $\liminf_{\scl \to \infty} f_\scl (t, x) \geq g(t, x)$ uniformly for $(t, x) \in \mathcal{O}$, if for any
$\z > 0$, there exists $M$ such that $f(t, x) >  g(t, x) - \z$ for all $\scl > M$ and $(t, x) \in \mathcal{O}.$ Similarly, we say that $\limsup_{\scl \to \infty} f_\scl (t, x) \leq g(t, x)$ uniformly for $(t, x) \in \mathcal{O}$ if for any
$\z > 0$, there exists $M$ such that $f(t, x) <  g(t, x) + \z$ for all $\scl > M$ and $(t, x) \in \mathcal{O}.$
\begin{proof}[Proof of Lemma \ref{lem:bboptime}]
The idea of the proof is from \cite[Theroem 4.1.6]{chen2010random}. 
Recall that for a Brownian bridge $\bb(s)$ with $\bb(0) = x$ and $\bb(\scl) = 0$, the random variable $\bb(1)$ and $\bb(\scl -1)$ has joint probability density function 
$$
f_{\bb(1), \bb(\scl-1)} (y, z) := \frac{p(1, x-y) p(\scl-2, y-z) p(1, z)}{p(\scl, x)}. 
$$ 
Using the boundedness of $\varphi$ and the 
preceding joint density in order, we have

	\begin{align}\notag
	&\E_{x \to 0} \bigg[\exp\Big(\int_0^{\scl} \pot(B_b(s)) ds\Big)\bigg]\\
	\notag
	 &\geq \frac{1}{C}\E_{\bb(0) = x, \bb(\scl) = 0} \bigg[\exp\Big(\int_1^{\scl-1} \pot(B_b(s)) ds\Big)\bigg]\\
	\label{eq:potlowerbd}
	&=
	\frac{1}{C p(\scl, x)} \int_\R p(1, x-y) \mathbb{E}_{y}\bigg[\exp\Big(\int_1^{\scl-1} \pot(B(s)) ds\Big) p(1, B(\scl -1))\bigg] dy. 
	\end{align}
Since $\pot$ is fixed, we omit the dependence on $\pot$ 
in the constant $C$. For any compactly supported $g \in \Cc(\R)$, since $|x| \leq \scl^\alpha$, $\frac{1}{p(\scl, x)} \geq C\sqrt{\scl}$ and $p(1, x-y) \geq \frac{1}{C(g)} |g(y)| \exp(-C \scl^{2\alpha})$. Hence, 
	\begin{align*}
	&\frac{1}{p(\scl, x)}\int_\R p(1, x-y) \mathbb{E}_{y}\bigg[\exp\Big(\int_1^{\scl-1} \pot(B(s)) ds\Big) p(1, B(\scl -1))\bigg] dy\\
	&\geq \frac{\exp(-C \scl^{2\alpha})}{C(g)}\int_\R g(y) \mathbb{E}_{y}\bigg[\exp\Big(\int_1^{\scl-1} \pot(B(s)) ds\Big) g(B(\scl -1))\bigg] dy\\
	&\geq \frac{ \exp(-C \scl^{2\alpha})}{C(g)} \langle T^\pot_{\scl - 2} g, g \rangle \geq \frac{\exp(-C \scl^{2\alpha})}{C(g)} \exp\Big((\scl - 2) \langle A^\pot g, g \rangle_{L^2(\R)}\Big).
	\end{align*}
The last inequality follows from \eqref{eq:chenresult}. Inserting the above bound into 
the right hand side of \eqref{eq:potlowerbd} yields 
\begin{equation*}
\E_{x \to 0} \bigg[\exp\Big(\int_0^{\scl} \pot(B_b(s)) ds\Big)\bigg] \geq \frac{\exp(-C \scl^{2\alpha})}{C(g)} \exp\Big((\scl - 2) \langle A^\pot g, g \rangle_{L^2(\R)}\Big).
\end{equation*}
Taking the logarithm of the both sides, dividing the result by $\scl$ and sending $\scl \to \infty$, we obtain that uniformly for $|x| \leq \scl^{\alpha}$, 
$$\liminf_{\scl \to \infty} \frac{1}{\lambda} \log \E_{x \to 0}\bigg[\exp\Big(\int_0^{\scl} \pot(B_b(s)) ds\Big)\bigg] \geq \langle A^\pot g, g \rangle_{L^2(\R)}.$$
Taking the supremum of the right hand side over $g \in \Cc(\R)$ and using \eqref{eq:ffuncF}, 
we conclude that uniformly for $|x| \leq \scl^\alpha$, 
	\begin{equation*}
	\liminf_{\scl \to \infty} \frac{1}{\lambda} \log \E_{x \to 0}\bigg[\exp\Big(\int_0^{\scl} \pot(B_b(s)) ds\Big)\bigg] \geq F(\pot).
	\end{equation*}
	To prove the reverse inequality, we write 
	\begin{align*}
	&\E_{x \to 0} \bigg[\exp\Big(\int_0^{\scl} \pot(B_b(s)) ds\Big)\bigg] \\
	&= \E_{x \to 0} \bigg[\exp\Big(\int_0^{\scl} \pot(B_b(s)) ds\Big) \ind_{\{|B_b(1)| \leq \scl^2, |B_b(\scl-1)| \leq \lambda^2 \}}\bigg]\\ 
	&\quad + \E_{x \to 0} \bigg[\exp\Big(\int_0^{\scl} \pot(B_b(s)) ds\Big) \ind_{\{|B_b(1)| \geq \scl^2 \text{ or } |B_b(\scl-1)|  \geq \lambda^2 \}}\bigg] \\
	\numberthis \label{eq:optimebd1}
	&\leq \E_{x \to 0} \bigg[\exp\Big(\int_0^{\scl} \pot(B_b(s)) ds\Big) \ind_{\{|B_b(1)| \leq \scl^2, |B_b(\scl-1)| \leq \lambda^2 \}}\bigg] + C \exp(-C^{-1} \lambda^2).
	\end{align*}
Recall that $|x|\leq \scl^{\alpha}$, $\alpha \in (0, \frac{1}{2})$. The last inequality follows from the boundedness of $\pot$ together with the tail decay of  $\P[\bb(1) \geq \scl^2]$ and $\P[\bb(\scl - 1) \geq \scl^2]$.  For the first term on the right hand side of \eqref{eq:optimebd1}, 
	\begin{align*}
	&\E_{x \to 0} \bigg[\exp\Big(\int_0^{\scl} \pot(B_b(s)) ds\Big) \ind_{\{|B_b(1)| \leq \scl^2, |B_b(\scl-1)| \leq \lambda^2 \}}\bigg]\\ 
	&\leq C\E_{\bb(0) = x, \bb(\scl) = 0} \bigg[\exp\Big(\int_1^{\scl-1} \pot(B_b(s)) ds\Big) \ind_{\{|B_b(1)| \leq \scl^2, |B_b(\scl-1)| \leq \lambda^2 \}}\bigg]
	\\
	\numberthis \label{eq:optimebd2}
	&= \frac{C}{p(\scl, x)} \int_\R p(1, x-y) \ind_{\{|y| \leq \scl^2\}} \mathbb{E}_{y}\bigg[\exp\Big(\int_1^{\scl-1} \pot(B(s)) ds\Big) p(1, B(\scl -1)) \ind_{\{|B(\scl-1)| \leq \scl^2\}}\bigg] dy.
	\end{align*}
	The first inequality follows from the boundedness of $\pot$ and the second equality is due to the transition probabilities of the Brownian bridge. Let $g_\scl \in \Cc(\R)$ be such that $g_\scl (y) = 1$ for all $|y| \leq \scl^2$, $g_\scl (y) = 0$ for $|y| \geq 2\scl^2$ and $g_\scl (y) \in [0, 1]$ for all $y$. Then $p(1, x-y) \ind_{\{|y| \leq \scl^2\}} \leq g_\scl (y)$ and thus
	\begin{align*}
	&\int p(1, x-y) \ind_{\{|y| \leq \scl^2\}} \mathbb{E}_{y}\bigg[\exp\Big(\int_1^{\scl-1} \pot(B(s)) ds\Big) p(1, B(\scl -1)) \ind_{\{|B(\scl-1)| \leq \scl^2\}}\bigg] dy\\
	&\leq \int g_\scl(y) \E_{y}\bigg[\exp\Big(\int_1^{\scl-1} \pot(B(s)) ds\Big) g_\scl(B(\scl-1))\bigg] dy
	\\
	&= \langle T^\pot_{\scl - 2} g_\scl, g_\scl\rangle
	_{L^2(\R)} \leq \normL{g_\lambda}^2 \exp\big((\scl-2) F(\pot)\big).
	\end{align*}	
	The last equality follows from \eqref{eq:chenresult}. Inserting the above bound into \eqref{eq:optimebd2} and then inserting the result to \eqref{eq:optimebd1}, we conclude that uniformly for $|x| \leq \scl^\alpha$,
	\begin{equation*}
	\limsup_{\scl \to \infty} \frac{1}{\lambda} \log \E_{x \to 0}\bigg[\exp\Big(\int_0^{\scl} \pot(B_b(s)) ds\Big)\bigg] \leq F(\pot).
	\end{equation*}  
	This concludes the lemma.
\end{proof}
\indentation We are now ready to prove \eqref{eq:limsuptail}. Recall that $\rho_*(t, x) = \devm(x) = \sech^2(x)$. 
Fix arbitrary $\z > 0$ and take $x = 0$ in Lemma \ref{lem:bboptime}. We have 
\begin{align}\label{eq:limsuptail1}
\begin{split}
	\lim_{\scl \to \infty}  \frac{1}{\scl} \log \Zfn((1+\z)\devm; 2\scl, 0)  
	&=
	\lim_{\scl \to \infty}  \frac{1}{\lambda} \log \E_{0 \to 0}\bigg[\exp\Big(\int_0^{2\scl} (1+\z)\devm(B_b(s)) ds\Big)\bigg] \rev{\cdot \hk(2\scl,0)}
\\
	&= 2F((1+\z)\devm).
\end{split}
\end{align}
Referring to \eqref{eq:funcF} for the definition of $F$ and taking $g = \frac{1}{\sqrt{2}} \sech(x)$, $\|g\|_{L^2(\R)} = 1$, we have
\begin{equation*} 
F((1+\z)\devm) \geq \int_\R (1+\z) \devm(x) g(x)^2 - \frac{1}{2} g'(x)^2 dx = \frac{1}{2} + \frac{2}{3} \z. 
\end{equation*}
This, together with \eqref{eq:limsuptail1}, implies that for $\lambda $ large enough, $\Zfn((1+\z)\devm; 2\scl, 0) \geq e^{(1+\z) \lambda} > \frac{1}{\sqrt{4\pi \scl}} e^{\lambda}$. Consequently, for $\lambda$ large enough,  
\begin{equation*}
\inf \Big\{\frac{1}{2\scl}\|\dev\|_{L^2([0, 2\scl] \times \R)}^2: \Zfn(\dev; 2\scl, 0) \geq \frac{1}{\sqrt{4\pi \scl}} e^\scl\Big\} \leq \frac{1}{2\lambda} \|(1+\z) \devm\|_{L^2([0, 2\scl] \times \R)}^2 = \frac{4}{3}(1+\z)^2.
\end{equation*}
In the last equality, we used $\|\sech(\Cdot)^2\|_{L^2(\R)}^2 = \frac{4}{3}$. Taking $\z \to 0$ concludes \eqref{eq:limsuptail}.
\subsection{Proof of \eqref{eq:liminftail}}
\indentation 
Fix $T > 0$ and $\rho \in L^2([0, T] \times \R)$. Let us define the time-dependent semigroup $P(\dev; s\to t): L^2(\R) \to L^2(\R)$, $0 \leq s < t \leq T$. Given a function $f \in L^2(\R)$, fix $s$ and consider the PDE
\begin{equation}\label{eq:pde}
\partial_r\Zfn^{s, f} (r, x) = \frac{1}{2} \partial_{xx} \Zfn^{s, f} (r, x) + \dev(r, x) \Zfn^{s, f} (r, x), \qquad \Zfn^{s, f} (s, \Cdot) = f.
\end{equation} 
We define $P(\dev; s \to t) f := \Zfn^{s, f}(t, \Cdot)$. 
Note that $\Zfn^{s, f}$ is the unique solution to the integral equation
\begin{equation*}
\Zfn^{s, f} (t, x) = \int_\R p(t-s, x-y) f(y) dy + \int_s^t \int_\R p(t-r, x-y) \dev(r, y) \Zfn^{s, f} (r, y) drdy.
\end{equation*}  
Via Picard iteration, we have $\Zfn^{s, f}(t, x) = \int_{\mathbb{\R}} P((s,x) \to (t, y)) f(y) dy$ where the kernel 
\begin{align}
\notag
&P(\dev; (s, x) \to (t, y))\\ 
\label{eq:semigroup}
&:= p(t-s, x-y) + \sum_{n=1}^\infty \int_{s < t_n < \dots < t_1 < t} \int_{\R^n} \prod_{i=1}^{n+1} p(t_{i-1} - t_i, x_{i-1} - x_i) \dev(t_i, x_i) dt_i dx_i. 
\end{align}
Here, we set $t_0 = t$ and $x_0 = x$, $t_{n+1} = s$ and $x_{n+1} = y$.
\bigskip
\\
Next we establish four bounds for the time-dependent semigroup. 
\begin{lem}\label{l.ptL2}
There exists a universal constant $C$ such that for all $s < t$ and $a > 0$, 
\begin{align}
\label{e.l.ptL2}
&\|P(\dev; (s, 0) \to (t, \Cdot))\|_{L^2(\R)} \leq C(t-s)^{-1/4} \exp(a^2 C (t-s) + \frac{1}{a} \|\dev\|^2_{L^2([s, t] \times \R)}),\\
\label{e.l.L2pt}
&\|P(\dev; (s, \Cdot) \to (t, 0)\|_{L^2(\R)} \leq C(t-s)^{-1/4} \exp(a^2 C (t-s) + \frac{1}{a} \|\dev\|^2_{L^2([s, t] \times \R)}),\\
\label{e.l.L2L2}
&\|P(\rho; s \to t)\|_{L^2(\R) \to L^2(\R)} \leq  C\exp(a^2 C (t-s) + \frac{1}{a} \|\dev\|^2_{L^2([s, t] \times \R)}),\\
\label{e.l.ptpt}
&|P(\rho; (s, y) \to (t, x))| \leq C \exp(a^2 C(t-s) + \frac{1}{a} \|\dev\|^2_{L^2([s, t] \times \R)}) p(t-s, x-y).
\end{align}
\end{lem}
\begin{proof}
In the proof, we assume without loss of generality that $s = 0$.
\bigskip
\\ 
Let us first prove \eqref{e.l.ptL2}.
View the right hand side of \eqref{eq:semigroup} as a series of function in $x$, and bound the $L^2$ norm of each function. For the zeroth function in \eqref{eq:semigroup}, we have $\|p(t, \Cdot)\|_{L^2(\R)}^2 = C t^{-\frac{1}{2}}$. 
For the $n$-th functions in \eqref{eq:semigroup}, applying the Cauchy-Schwarz inequality gives 
\begin{align*}
\normL{\text{(n-th)}}^2 &\leq \int_\R  \int_{0 < t_n < \dots < t_1  < t} \int_{\R^n} \prod_{i=1}^{n+1} p(t_{i-1}-t_i, x_{i-1} - x_i)^2 dt_i dx_i dx \frac{\|\dev\|^{2n}_{L^2([0, t] \times \R)}}{n!}\\
&= \frac{C^n}{\Gamma(n/2)} t^{\frac{n}{2} - \frac{1}{2}}   \frac{\|\dev\|^{2n}_{L^2([0, t] \times \R)}}{n!}.
\end{align*}
Hence, we have 
\begin{align*}
\normL{P(\dev; (0, 0) \to (t, \Cdot))} &\leq
C t^{-\frac{1}{4}} \sum_{n=0}^\infty \frac{C^n t^{n/4}}{\Gamma(n/2)^{\frac{1}{2}}} \frac{\|\dev\|^n_{L^2([0, t] \times \R)}}{(n!)^{1/2}}\\ 
&\leq C t^{-\frac{1}{4}} \Big(\sum_{n = 0}^\infty \frac{a^n C^n t^{n/2}}{\Gamma(n/2)}\Big)^{\frac{1}{2}} \Big(\sum_{n = 0}^\infty \frac{a^{-n} \|\dev\|^{2n}_{L^2([0, t] \times \R)}}{n!} \Big)^{\frac{1}{2}}.
\end{align*}
This concludes \eqref{e.l.ptL2}.
\bigskip
\\
The left hand sides of \eqref{e.l.ptL2} and \eqref{e.l.L2pt} are the same upon time reversal, so \eqref{e.l.L2pt} follows.
\bigskip
\\
The proof of \eqref{e.l.L2L2} is similar in  spirit to \eqref{e.l.ptL2}. Using $(P(\dev; 0 \to t) f)(x) = \int_\R P((0, x) \to (t, y)) f(y) dy$ and \eqref{eq:semigroup}, we  express $P(\dev; 0 \to t) f$ as a series of functions. We bound the $L^2$ norm of each function in the series. 
By the Cauchy-Schwarz inequality, we have 
\begin{equation}\label{eq:basicineq}
\Big(\int_\R p(t, x-y) f(y) d y\Big)^2 \leq \int_\R p(t, x-y) f(y)^2 dy. 
\end{equation}
Integrating both sides of \eqref{eq:basicineq} in $x$ gives that for the zero-th function, $\normL{\int_\R p(t, \Cdot-y) f(y)dy}^2 \leq \normL{f}^2$.
For the $n$-th function, applying the Cauchy-Schwarz inequality yields
\begin{align*}
\normL{\text{(n-th)}}^2 \leq \int_\R \int_{0 < t_n < \dots < t_1 < t} \int_{\R^n} \Big(\int p(t_n, x_n - y) f(y) dy\Big)^2\prod_{i=1}^n p(t_{i-1} - t_i, x_{i-1} - x_i)^2 dt_i dx_i dx \frac{\|\dev\|_{L^2([0, t] \times \R)}^{2n}}{n!}.
\end{align*}
Using \eqref{eq:basicineq} and $p(t, x)^2 \leq \frac{1}{\sqrt{2\pi} t} p(t, x)$, we get
\begin{equation*}
\normL{\text{(n-th)}}^2 \leq \frac{\|\rho\|^{2n}_{L^2([0, t] \times \R)} \normL{f}}{n!} \int_{0 < t_n < \dots < t_1 < t} \prod_{i=1}^n \frac{1}{\sqrt{2\pi(t_{i-1} - t_i)}} dt_i = \frac{C^n t^{n/2}}{\Gamma(n/2)} \frac{\|\rho\|^{2n}_{L^2([0, t] \times \R)}}{n!} \|f\|^2_{L^2(\R)}.
\end{equation*}
Taking the square root of both sides and summing over $n$, the rest of the proof is  similar to the proof of \eqref{e.l.ptL2}.
\bigskip
\\
To prove \eqref{e.l.ptpt}, we view the right hand side of \eqref{eq:semigroup} as a series of real numbers. We bound the value of each term in the series. The zeroth term equals $p(t-s, x-y)$. For the $n$-th term, applying the Cauchy-Schwarz inequality yields
\begin{align*}
|(\text{n-th})|^2 \leq \int_{0 < t_n < \dots < s_1 < t} \int_{\R^n} \prod_{i=1}^{n+1} p(t_{i-1} - t_i, x_{i-1} - x_i)^2  dt_i dx_i \frac{\|\dev\|_{L([s, t] \times \R)}^{2n}}{n!} \leq \frac{C^n t^{n/2}}{\Gamma(n/2)} \frac{\|\dev\|^{2n}_{L^2([0, t] \times \R)}}{n!} p(t, x-y)^2.
\end{align*}
Taking the square root of both sides and summing over $n$, the rest of the proof is  similar to the proof of \eqref{e.l.ptL2}.

\end{proof}


\begin{lem}\label{lem:potbd}
	For every $\pot \in L^2(\R)$, we have 
	\begin{equation}\label{eq:funcFbd}
	F(\pot) \leq \frac{1}{2}(\frac{3}{4})^{2/3} \normL{\pot}^{\frac{4}{3}}.
	\end{equation}
Moreover, 
the equality holds if and only if $\pot(\Cdot) = \alpha^2 \sech^2(\alpha(\Cdot - v))$ for some $\alpha \geq 0$.
\end{lem}

\begin{proof}
Referring to \eqref{eq:funcF}, we have
\begin{align}\label{eq:potbd1}
F(\pot) &\leq \sup\Big\{ \normL{\pot} \Big(\int g(x)^4 dx\Big)^\frac{1}{2} - \frac{1}{2}\int g'(x)^2 dx\Big\} : g\in H^1(\R), \normL{g} = 1 \Big\}\\
\label{eq:potbd2}
&\leq \sup\Big\{3^{-\frac{1}{4}} \normL{\pot}  \normL{g'}^{\frac{1}{2}} - \frac{1}{2} \normL{g'}^2: g \in H^1(\R), \normL{g} = 1\Big\}
\\
\label{eq:potbd3}
&\leq \frac{1}{2}(\frac{3}{4})^{2/3} \normL{\pot}^{\frac{4}{3}}. 
\end{align}
The inequality \eqref{eq:potbd1} is due to the Cauchy-Schwarz inequality. The inequality \eqref{eq:potbd2} follows from the one-dimensional $L^4$ Gagliardo-Nirenberg-Sobolev inequality, which states that, for $g \in L^2(\R)$ and $g' \in L^2(\R)$, 
\begin{equation}\label{eq:potbd4}
\|g\|_{L^4(\R)} \leq 3^{-1/8} \normL{g'}^{1/4} \normL{g}^{3/4}.
\end{equation} 
The inequality \eqref{eq:potbd3} follows from the elementary inequality $3^{-\frac{1}{4}} y x^{\frac{1}{2}} - \frac{1}{2} x^2 \leq \frac{1}{2}(\frac{3}{4})^{2/3} y^{\frac{4}{3}}$.
\bigskip
\\
To prove the if and only if statement, we need to investigate when the inequalities \eqref{eq:potbd1}-\eqref{eq:potbd3} become equalities. 
The inequality \eqref{eq:potbd1} becomes an equality if and only if $\pot = \text{const}\cdot g$. By \cite[Proposition 3.1]{dolbeault2014one}, the inequality \eqref{eq:potbd2}, i.e. \eqref{eq:potbd4} becomes an equality if and only if $g(x) = a\sech(b(x-v))$ for some $a, v$ and $b > 0$; the inequality \eqref{eq:potbd3} becomes an equality if and only if $\normL{g'} = (\frac{1}{2})^{2/3} 3^{-1/6} \normL{\pot}^{2/3}$. Combining these with $\normL{g} = 1$  gives that 
\begin{equation*}
g(\Cdot) = a\sech\big(2a^2(\Cdot - v)\big), \qquad \pot(\Cdot) = 4a^4 \sech^2\big(2a^2(\Cdot - v)\big).
\end{equation*} 
Replacing $2a^2$ with $\alpha$ implies that $\pot(\Cdot) = \alpha^2 \sech^2(\alpha(\Cdot - v))$.
\end{proof}
\begin{rmk}
We only need \eqref{eq:funcFbd} for this section, and the if and only if part is required in Section \ref{sec:minimizer}. 
\end{rmk}
\indentation The key for proving \eqref{eq:liminftail} is an improved upper bound (compared with \eqref{e.l.L2L2}) for $\|P(\rho; s \to t)\|_{L^2(\R) \to L^2(\R)}$ stated in Proposition \ref{p.l2tol2}. To prove this upper bound, we need to first establish the following two lemmas about continuity. 
\begin{lem}\label{lem:continuity}
Fix $0 \leq s < t$. The map $\|P(\Cdot; s \to t)\|_{L^2(\R) \to L^2(\R)}: L^2([s, t] \times \R) \to \R$ is continuous.
\end{lem}
\begin{proof}
	
	Fix $f \in L^2(\R)$. Since $\Zfn^{s, f}(t, x) = (P(\dev; s \to t) f)(x)$ solves \eqref{eq:pde}, by the Feynman-Kac formula, 
	\begin{equation*}
	\Zfn^{s, f}(\dev; t, x) = \E_{x}\bigg[\exp\Big(\int_s^t \dev(t-r, B(r)) dr\Big) f(B(t))\bigg].
	\end{equation*}
	Given $\dev, \tdev \in L^2([s, t] \times \R)$. We write $\dev = \dev - (1-\z)\tdev + (1-\z) \tdev$ in the exponent above and apply the H\"{o}lder inequality 
	to the result. We have for $\z \in (0, 1)$,
	\begin{align*}
	\Zfn^{s, f} (\dev; t, x) &= \E_{x}\bigg[\exp\Big(\int_s^t (\dev - (1-\z) \tdev)(t-r, B(r)) dr\Big) \big(f(B(t))\big)^{\z}  \\
	&\hspace{5.2em} \cdot \exp\Big(\int_s^t (1-\z) \tdev(t-r, B(r)) dr\Big) \big(f(B(t))\big)^{1-\z}\bigg]\\
	&\leq \Zfn^{s, f}\big(\z^{-1}(\dev - (1-\z) \tdev); t, x\big)^{\z} \Zfn^{s, f}(\tdev; t, x)^{1-\z}.
	\end{align*}
Squaring both sides, integrating in $x$ and using the H\"{o}lder inequality again gives
	\begin{equation*}
	\big\|\Zfn^{s, f}(\rho; t, \Cdot)\big\|_{L^2(\R)} \leq \big\|\Zfn^{s, f}\big(\z^{-1}(\dev - (1-\z) \tdev); t, \Cdot\big)\big\|_{L^2(\R)}^\z \big\|\Zfn^{s, f}(\tdev; t, \Cdot)\big\|_{L^2(\R)}^{1-\z}.
	\end{equation*}
	Taking the supremum over $\{f: \|f\|_{L^2(\R)} \leq 1\}$ yields that
	\begin{equation}\label{e.l.continuity1}
	\big\|P(\dev; s \to t)\big\|_{L^2(\R) \to L^2(\R)} \leq \big\|P\big(\z^{-1}(\dev - (1-\z) \tdev); s \to t\big)\big\|_{L^2(\R) \to L^2(\R)}^\z \big\|P(\tdev; s \to t)\big\|_{L^2(\R) \to L^2(\R)}^{1-\z}.
	\end{equation}
Set $G(\dev):= \log \|P(\dev; s \to t)\|_{L^2(\R) \to L^2(\R)}$. Our goal is to show that $G$ is continuous in $\dev$. We do this by showing that 
$\GG$ is both upper and lower semi-continuous. Taking the logarithm of both sides of \eqref{e.l.continuity1} and applying \eqref{e.l.L2L2} yields 
	\begin{align}
	\notag
	\GG(\dev) &\leq \z \GG\big(\z^{-1}(\dev - (1-\z) \tdev)\big) + (1-\z) \GG(\tdev)\\
	\label{e.l.continuity2}
	&\leq C(s, t) \z  + 
	C \z \|\z^{-1}(\dev - (1-\z) \tdev)\|_{L^2([s, t] \times \R)}^2
	+ (1-\z)G(\tdev).
	\end{align}
	Subtracting $(1-\z) \GG(\tdev) + \z \GG(\dev)$ from both sides and applying the inequality $\|f + g\|_{L^2}^2 \leq 2(\|f\|^2_{L^2} + \|g\|_{L^2}^2)$ yield
	\begin{align*}
	(1-\z) (\GG(\dev) - \GG(\tdev)) &\leq  C(s, t)  \z + C \z^{-1} (1-\z)^2 \|\dev - \tdev\|_{L^2([s, t] \times \R)}^2 + C\z \|\dev\|_{L^2([s, t] \times \R)}^2 - \z \GG(\dev)\\
	&\leq C\z^{-1} (1-\z)^2 \|\dev - \tdev\|_{L^2([s, t] \times \R)}^2 + \z C(s, t, \dev).
	\end{align*}
	Taking $\z = \|\dev - \tdev\|_{L^2([s, t] \times \R)}$ and letting $\tdev \to \dev$, we see that $\limsup_{\tdev \to \dev} \GG(\tdev) \leq G(\dev)$. Thus, $G$ is upper semi-continuous.
	\bigskip
	\\
	On the other hand, if we swap $\dev$ and $\tdev$ in \eqref{e.l.continuity2}, a similar argument gives 
\begin{equation*}
G(\tdev) - G(\dev)
\leq \z C(s, t, \dev) + C \z^{-1} \|\tdev - \dev\|_{L^2([s, t] \times \R)}^2.
\end{equation*}
	Taking $\z = \|\dev - \tdev\|_{L^2([s, t] \times \R)}$ and letting $\tdev \to \dev$, we see that $\GG$ is lower semi-continuous. This concludes the lemma.
\end{proof}
\begin{lem}\label{lem:Fcont}
Fix $0 \leq s<  t$ and $M \geq 0$. There exists a constant $C(M) > 0$. For $\varphi_1, \varphi_2 \in L^2(\R)$ such that $0 \leq \varphi_1, \varphi_2  \leq M$ almost everywhere, we have 
	\begin{equation*}
	|F(\varphi_1) - F(\varphi_2)| \leq C(M) \|\varphi_1 - \varphi_2\|_{L^2(\R)}.
	\end{equation*}
\end{lem}
\begin{proof}
	By the definition of $F(\varphi_2)$ in \eqref{eq:funcF}, for any $\z > 0$, there exists $g_\z$ that $\|g_\z\|_{L^2(\R)}^2  =1$ and 
	\begin{equation}\label{eq:Fcont1}
	F(\varphi_2) \leq \int_\R \varphi_2 g_\z^2  - \frac{1}{2} (g'_\z)^2 + \z. 
	\end{equation}
	Moreover, we have $F(\varphi_1) \geq \int_\R \varphi_1 g_\z^2 - \frac{1}{2} (g'_\z)^2$. Using this and \eqref{eq:Fcont1}, we have 
\begin{align*}
F(\varphi_2)  - F(\varphi_1) \leq \int_\R (\varphi_2 - \varphi_1) g_\z^2 + \z. 	
\end{align*}
Applying the Cauchy-Schwarz inequality to the right hand side above and then applying \eqref{eq:potbd4} to the result, we have
\begin{equation}\label{eq:Fcont2}
F(\varphi_2)  - F(\varphi_1) \leq 3^{-\frac{1}{4}} \|\varphi_2 - \varphi_1\|_{L^2(\R)} \|g_\z'\|_{L^2(\R)}^{\frac{1}{2}} + \z.
\end{equation}
By \eqref{eq:Fcont1}, $\frac{1}{2}\|g'_\z\|_{L^2(\R)}^2 \leq \int_\R \varphi_2 g_\z^2  - F(\varphi_2) + \z.$ 
Since $0 \leq \varphi_2 \leq M$ and $\|g_\z\|_{L^2(\R)} = 1$, we can upper bound the first term on the right hand side by $M$ and the second term by a universal constant $C$. Hence, we have 
$\frac{1}{2}\|g'_\z\|_{L^2(\R)}^2 
\leq  C(M).$ Inserting this to the right hand side of \eqref{eq:Fcont2} and letting $\z \to 0$ yields   
	\begin{equation*}
	F(\varphi_2) - F(\varphi_1) \leq  C(M) \|\varphi_2 - \varphi_1\|_{L^2(\R)}.
	\end{equation*}
	Swapping $\varphi_1$ and $\varphi_2$ concludes the lemma. 
\end{proof}
\begin{cor}\label{cor:Fcont}
	Fix $0 \leq s < t$ and $M \geq 0$. There exists a constant $C(M, t, s)$. For $\rho_1, \rho_2 \in L^2([s, t] \times \R)$ such that $0 \leq \rho_1, \rho_2 \leq M$ almost everywhere, we have 
	\begin{equation*}
	\Big|\int_s^t F(\rho_1(r, \Cdot)) dr - \int_s^t F(\rho_2 (r, \Cdot)) dr\Big| \leq C(M, t, s) \|\rho_1 - \rho_2\|_{L^2([s, t] \times \R)}.
	\end{equation*}   
\end{cor}
\begin{proof}
It is straightforward to see that
	\begin{align}\label{eq:Fcont3}
	\Big|\int_s^t F(\rho_1(r, \Cdot)) dr - \int_s^t F(\rho_2 (r, \Cdot)) dr\Big| \leq \int_s^t \big| F(\rho_1 (r, \Cdot))  - F(\rho_2(r, \Cdot))\big| dr.
\end{align}
Applying Lemma \ref{lem:Fcont} to upper bound the right hand side above and applying the Cauchy Schwarz inequality to the result, we have
\begin{equation*}
\text{RHS of \eqref{eq:Fcont3}} \leq C(M) \int_s^t \|\rho_1(r, \Cdot) - \rho_2 (r, \Cdot)\|_{L^2(\R)} dr \leq C(M) (t-s)^{\frac{1}{2}} \|\rho_1 - \rho_2\|_{L^2([s, t] \times \R)}.
\end{equation*}
This concludes the corollary.
\end{proof}

\begin{prop}\label{p.l2tol2}
Fix $0 \leq s < t$. For all $\rho \in L^2([s, t] \times \R)$ satisfying $\rho \geq 0$ almost everywhere, we have
\begin{equation}\label{e.p.l2tol2}
\|P(\dev; s \to t)\|_{L^2(\R) \to L^2(\R)} \leq \exp\Big(\int_s^t F(\dev(r, \Cdot)) dr\Big). 
\end{equation}
\end{prop}
\begin{proof}
Fix smooth and compactly supported $\rho$ and $f$. Since $\rho$ is smooth, the function $\Zfn^{s, f}(t, x) = (P(s\to t) f)(x)$ solves \eqref{eq:pde} point-wise.
Multiply both side by $\Zfn$, integrate the result over $x \in \R$ and interchange the integration and differentiation. We have 
\begin{equation*}
\frac{1}{2}  \partial_r 	\normL{\Zfn^{s, f}(r, \Cdot)}^2 \leq F(\rho(r, \Cdot)) \normL{\Zfn^{s, f}(r, \Cdot)}^2.
\end{equation*}
Integrating in time from $s$ to $t$ gives 
\begin{equation*}
\normL{\Zfn^{s, f} (t, \Cdot)}^2 \leq \normL{f}^2 \exp\Big(2\int_s^t F(\rho(r, \Cdot)) dr\Big).
\end{equation*}
For fixed smooth and compact supported $\dev$, by approximation, the above inequality also holds for general $f \in L^2(\R)$.  
Thus, we have proven 
\eqref{e.p.l2tol2} for smooth and compactly supported $\dev$. 
\bigskip
\\
The general result follows by approximation and monotonicity.
Fix arbitrary $M \geq 0$. We first prove that \eqref{e.p.l2tol2} holds for all $\rho \in L^2([s, t] \times \R)$ such that almost everywhere $0 \leq \rho \leq M$. By a density argument, there exists smooth and compactly supported functions $\{\rho_n\}_{n=1}^\infty$ such that $0 \leq \rho_n \leq M$ and $\|\rho_n - \rho\|_{L^2([s, t] \times \R)} \to 0$ as $n \to \infty$. \eqref{e.p.l2tol2} holds when $\rho$ is replaced by $\rho_n$. Let $n \to \infty$, by Lemma \ref{lem:continuity} and Corollary \ref{cor:Fcont}, \eqref{e.p.l2tol2} also holds for $\rho$. 
\bigskip
\\ 
Finally, we prove \eqref{e.p.l2tol2} for  $\rho \in L^2([s, t] \times \R)$ such that almost everywhere $\rho 
\geq 0$. 
As $n \to \infty$, $\rho \mathbf{1}_{\{\rho \leq n\}} \to \rho$ in $L^2([s, t] \times \R)$. In the preceding paragraph, we have proved that \eqref{e.p.l2tol2} holds when $\rho$ is replaced by $\rho \mathbf{1}_{\{\rho \leq n\}}$, for all $n$. Using this and the monotonicity of $F$, we have 
\begin{equation*}
\|P(\rho \mathbf{1}_{\{\rho \leq n\}}; s \to t)\|_{L^2(\R) \to L^2(\R)} \leq \exp\Big(\int_s^t F(\rho \mathbf{1}_{\{\rho \leq n\}} (r, \Cdot)) dr\Big) \leq \exp\Big(\int_s^t F(\rho(r, \Cdot)) dr\Big).
\end{equation*}
Letting $n \to \infty$, by Lemma \ref{lem:continuity}, the left hand side above converges to $\|P(\dev; s \to t)\|_{L^2(\R) \to L^2(\R)}$. Hence, we conclude \eqref{e.p.l2tol2}.
\end{proof}
\indentation Based on the preceding results, we now proceed to prove \eqref{eq:liminftail}. The proof amounts to showing that, 
for any $\dev \in L^2([0, 2\scl] \times \R)$ such that $\Zfn(\dev; 2\scl, 0) \geq \frac{1}{\sqrt{4\pi}} e^\scl$, we have $\frac{1}{2\scl}\|\dev\|^2_{L^2([0, 2\scl] \times \R)} \geq \frac{4}{3} - o_\lambda (1)$ where $\lim_{\scl \to \infty} o_\scl (1)  = 0$. 
\bigskip
\\
With loss of generality, we can assume that \begin{equation}\label{eq:l2assumption}
\frac{1}{2\scl} \|\rho\|^2_{L^2([0, 2\scl]) \times \R} \leq 2.
\end{equation}
Moreover, define $\rho_+ := \min
(\rho, 0)$. Because $\Zfn(\rho_+; 2\scl, 0) \geq \Zfn(\rho; 2\scl, 0)$ and $\|\rho_+\|_{L^2([0, 2\scl] \times \R)} \leq \|\rho\|_{L^2([0, 2\scl] \times \R)}$, we can also assume that $\rho \geq 0$ almost everywhere. 
Note that $\Zfn(\dev; 2\scl, 0) = P(\dev; (0, 0) \to (2\scl, 0))$. Using the semigroup property, we have
\begin{align*}
\Zfn(\dev; 2\scl, 0) &= P(\dev; (0, 0) \to (1, \Cdot)) P(\dev; 1 \to (2\lambda - 1)) P(\dev; (2\lambda - 1, \Cdot) \to (2\lambda, 0))\\
&\leq \normL{P(\dev; (0, 0) \to (1, \Cdot))} \|P(\dev; 1 \to (2\lambda - 1))\|_{L^2(\R) \to L^2(\R)} \normL{P(\dev; (2\lambda - 1, \Cdot) \to (2\lambda, 0))}.
\end{align*}  
Applying \eqref{e.l.ptL2}-\eqref{e.l.L2pt} (take $a = \scl^{\frac{1}{3}}$) to upper bound the first and third term  on the right hand side, using Proposition \ref{p.l2tol2} and \eqref{eq:l2assumption} to upper bound the second term, we have 
\begin{align}
\notag
\Zfn(\dev; 2\scl, 0) 
&\leq  C \|P(\dev; 1 \to (2\lambda - 1))\|_{L^2(\R) \to L^2(\R)}   \exp\big(C\scl^{\frac{2}{3}} + \scl^{-\frac{1}{3}} \|\rho\|^2_{L^2([0, 1] \times \R)} + \scl^{-\frac{1}{3}}\|\rho\|^2_{L^2([2\lambda - 1, 2\lambda] \times \R)}\big) \\
\label{e.t.uptail}
&\leq C \exp(C\scl^{\frac{2}{3}}) \exp\Big(\int_0^{2\scl} F(\rho(r, \Cdot)) dr\Big).
\end{align}
Here $C$ is a universal constant. For the integral above, using Lemma \ref{lem:potbd} to bound $F$, together with Young's inequality, $\normL{\dev(r, \Cdot)}^{\frac{4}{3}} \leq \frac{6^{1/3}}{3}(\normL{\dev(r, \Cdot)}^2 + \frac{2}{3})$ for every $r$, we have
\begin{align*}
\Zfn(\dev; 2\scl, 0) &\leq C\exp(C\scl^{\frac{2}{3}}) \exp\Big(\int_0^{2\scl} \frac{1}{2}\Big(\frac{3}{4}\Big)^{\frac{2}{3}} \normL{\dev(r, \Cdot)}^\frac{4}{3} dr\Big)\\
&\leq C\exp(C\scl^{\frac{2}{3}}) \exp\bigg(\int_0^{2\scl} \frac{1}{4}\Big(\normL{\dev(r, \Cdot)}^2 + \frac{2}{3}\Big) dr\bigg).
\end{align*}
Applying $\frac{1}{\sqrt{4\pi \scl}} e^{\lambda} \leq \Zfn(\dev; 2\scl, 0)$ to lower bound $\Zfn(\dev; 2\scl, 0)$, we have that
that $\scl(1-o_\scl(1)) \leq \frac{1}{4} \|\dev\|^2_{L^2([0, 2\scl] \times \R)} + \frac{\scl}{3}$, where $\lim_{\scl \to }o_\scl (1) = 0$. Hence $\frac{1}{2\scl}\|\dev\|^2_{L^2([0, 2\scl] \times \R)} \geq \frac{4}{3} - o_\lambda (1)$. Letting $\scl \to \infty$ concludes the desired \eqref{eq:liminftail}.                 
\section{Minimizers of the variational formula}\label{sec:minimizer}
\indentation For a $\dev \in L^2([0, 2\scl] \times \R)$, recall that $\Zfn(\dev) = \Zfn(\rho; t, x)$ is the (unique) solution to the PDE
\begin{equation*}
\partial_t \Zfn = \frac{1}{2}\partial_{xx} \Zfn + \rho \Zfn, \qquad (t, x) \in (0, 2\scl] \times\R, \qquad \Zfn(\dev; 0, \Cdot) = \delta_0 (\Cdot).
\end{equation*}
Let $\calK_\scl$ denote the set of minimizers of 
\begin{equation}\label{eq:minimizer}
\inf\Big\{\frac{1}{2}\|\dev\|_{L^2([0, 2\scl] \times \R)}^2:\dev \in L^2([0, 2\scl] \times \R), \, \Zfn(\dev; 2 \scl, 0) \geq \frac{1}{\sqrt{4\pi \scl}} e^\scl\Big\}.
\end{equation}
It is convenient to consider the scaled version of $\calK_\scl$: Let $\tcalK_\scl$ be the set of minimizers of 
\begin{equation}\label{eq:tminimizer}
\inf\Big\{\frac{1}{2}\|\dev\|_{L^2([0, 2] \times \R)}^2: \dev \in L^2([0, 2] \times \R),\, \Zfn(\dev; 2, 0) \geq \frac{1}{\sqrt{4\pi}} e^\scl\Big\}.
\end{equation}
Recall the scaling property 
$\Zfn(\scl\dev(\scl \Cdot, \scl^{\frac{1}{2}} \Cdot); 2, 0) = \scl^{\frac{1}{2}} \Zfn(\dev; 2\scl, 0)$. We have
\begin{equation}\label{eq:calKscaling}
\tcalK_\scl = \{\scl \dev(\scl\Cdot, \scl^{\frac{1}{2}}\Cdot): \dev \in \calK_\scl\}.
\end{equation}
To prepare for the proof of Theorem \ref{thm:limshape}, in this section, we develop a few properties of $\calK_\scl$ for fixed $\scl$. We also study the asymptotic behavior of $\calK_\scl$ as $\scl \to \infty$.
\subsection{Properties of $\calK_\scl$}
\label{sec:calK}
\subsubsection{$\calK_\scl$ is not empty.}
We first show that $\calK_\scl$ is not empty. By the scaling \eqref{eq:calKscaling}, it suffices to show that $\widetilde{\calK}_\scl$ is not empty, i.e. a minimizer of \eqref{eq:tminimizer} exists.
\bigskip
\\
Let us first introduce a few facts that would be useful throughout Section \ref{sec:minimizer}. We take 
a Banach space $\bana$ with Gaussian measure $\mu$ such that $L^2([0, 2] \times \R) $ is the Cameron-Martin space of $(\bana, \mu)$. 
\rev{We further assume the embedding $L^2([0, 2] \times \R) \subseteq \bana$ is dense.}
For a concrete choice of $(\bana, \mu)$, see \cite[Section 2.1.1]{lin2021short}.
The advantage of introducing $\bana$ is that it brings us \emph{compactness}.
Since $L^2([0, 2] \times \R)$ is the Cameron-Martin space of $\bana$, for arbitrary $r \geq 0$, the set $\{\dev: \|\dev\|_{L^2([0, 2] \times \R)} \leq r\}$ is compact in $\bana$. The compactness would be crucial for us to prove the non-emptiness of $\calK_\scl$.
\bigskip
\\
We are going to view $\Zfn$ as a map from $\dev \in L^2([0, 2] \times \R)$ to the space-time function $\Zfn(\dev; \Cdot, \Cdot)$. If we restrict the time and space coordinates in $(t, x) \in [\delta, 2] \times [-\delta^{-1}, \delta^{-1}]$, 
it follows from \cite[Section 2.1]{lin2021short} that $\Zfn$ maps $L^2([0, 2] \times \R)$ to $C([\delta, 2] \times [-\delta^{-1}, \delta^{-1}])$.
\bigskip
\\
We use $L^{\rev{2}}_\bana([0, 2] \times \R)$ to denote the same space $L^2([0,2] \times 
\R)$ with the topology induced by the topology of the Banach space $\bana$. 
\rev{Even though $ \Zfn: L^2_\bana([0, 2] \times \R) \to C([\delta, 2] \times [-\delta^{-1}, \delta^{-1}])$ is not continuous in general, its restriction onto $ \{ \rho: \|\rho\|_{L^2([0, 2] \times \R)} \leq r \} $ is continuous for any $ r<\infty $, as shown in the following lemma.}

\begin{lem}\label{lem:equicontinuity}
Fix arbitrary $\delta > 0$ and $ r<\infty $. The map $\Zfn: \rho \mapsto \Zfn(\rho; \Cdot, \Cdot)$ \rev{is} continuous from $L_\bana^2([0, 2] \times \R) \rev{\cap \{ \rho: \|\rho\|_{L^2([0, 2] \times \R)} \leq r \}}$  
to $C([\delta, 2] \times [-\delta^{-1}, \delta^{-1}])$.
\end{lem}
\begin{proof}
Define the extension of $\Zfn$
\begin{equation}\label{eq:Zprime}
\Zfn': \bana \to C([\delta, 2] \times 
[\delta^{-1}, \delta^{-1}]),\qquad \Zfn'(\dev) :=
\begin{cases}
\Zfn(\rho), & \text{ when } \rho \in L^2([0, 2] \times \R), \\
0, & \text{ otherwise.}
\end{cases} 
\end{equation}
By the proof of \cite[Lemma 3.7]{lin2021short} (see the paragraph above Eq (3.18') therein), there exists a sequence of continuous functions $\varphi_N: \bana \to C([\delta, 2] \times [-\delta^{-1}, \delta^{-1}])$ such that for all $r < \infty$,
\begin{equation*}
\lim_{N \to \infty} \sup_{\frac{1}{2}\|\rho\|^2_{L^2([0, 2] \times \R)} \leq r} \|\Zfn'(\dev) - \varphi_N (\rho)\|_{L^\infty([\delta, 2] \times [-\delta^{-1}, \delta^{-1}])} = 0.
\end{equation*}
Using this, we conclude that the map $\Zfn': \bana \rev{\cap \{ \rho: \|\rho\|_{L^2([0, 2] \times \R)} \leq r \}} \to C([\delta, 2] \times [-\delta^{-1}, \delta^{-1}])$ is continuous. By \eqref{eq:Zprime}, we have the desired continuity of $\Zfn: L_\bana^2([0, 2] \times \R) \rev{\cap \{ \rho: \|\rho\|_{L^2([0, 2] \times \R)} \leq r \}} \to C([\delta, 2] \times [-\delta^{-1}, \delta^{-1}])$.
\end{proof}
\indentation Denote the infimum in \eqref{eq:tminimizer} by $\infi$.
\begin{prop}\label{prop:existence}
	Fix $\scl > 0$, the set $\calK_\scl$ is not empty.
\end{prop}
\begin{proof}
	By the scaling \eqref{eq:calKscaling}, it suffices to prove that $\tcalK_\scl$ is not empty, i.e.\ the minimizer of \eqref{eq:tminimizer} exists. 
	Let $\{\rho_n\}_{n \in \Z_{\geq 1}} \subseteq L^2([0, 2] \times \R)$ be such that $\Zfn(\dev_n; 2, 0) \geq \frac{1}{\sqrt{4\pi}} e^{\scl}$ and $\frac{1}{2}\|\dev_n\|_{L^2([0, 2] \times \R)}^2 \downarrow \infi$. Since for arbitrary fixed $r \rev{<\infty} $, the set $\{\rho: \frac{1}{2}\|\rho\|^2_{L^2([0, 2] \times \R)} \leq r\}$ is compact in $\bana$, after passing to a subsequence, $\dev_n$ converges to some $\rho \in L^2([0, 2] \times \R)$ in the topology of $\bana$.  
	By Lemma \ref{lem:equicontinuity}, $\dev \mapsto \Zfn(\dev; 2, 0)$ is a continuous map from $L_\bana([0, 2] \times 
	\R)\rev{\cap \{ \rho: \|\rho\|_{L^2([0, 2] \times \R)} \leq r \}}$ to $\R$. This implies that $\Zfn(\dev; 2, 0) \geq 
	\frac{1}{\sqrt{4\pi}} e^{\scl}$. 
By the compactness of $\{\rho: \frac{1}{2}\|\rho\|^2_{L^2([0, 2] \times \R)} \leq r\}$, we have 
	$\|\dev\|_{L^2([0, 2] \times \R)} \leq \liminf_{n \to \infty} \|\dev_n\|_{L^2([0, 2] \times \R)}$. Hence, $\rho$ is a minimizer of \eqref{eq:tminimizer}. 
\end{proof}
\subsubsection{Symmetric decreasing function in space}\label{sec:sdf}
In this subsection, we show that every $\rho \in \calK_\scl$ is symmetric decreasing in space.
\bigskip
\\
We say that $f$ is \emph{symmetric decreasing} if there exists a non-increasing function $g: [0, \infty) \to [0, \infty)$ such that $f(x) = g(|x|)$ holds for almost every $x \in \R$. We say that $f$ is \emph{strictly symmetric decreasing} if $g$ is strictly decreasing. We say that a measurable function $\dev: [0, T] \times \R \to 
[0, \infty)$ is \emph{symmetric decreasing in space} if for almost every $0 \leq s \leq T$, $\rho(s, \Cdot)$ is symmetric decreasing.
\begin{prop}\label{prop:SD}
Fix $\scl > 0$,	every element $\dev \in \calK_\scl$ is symmetric decreasing in space. 
\end{prop}
\indentation Let $A \subseteq \R$ be a measurable set with finite Lebesgue measure $|A|$. We define the \emph{symmetric rearrangement} of $A$, denoted $A^*$, as the interval $[-|A|/2, |A|/2]$. Given a measurable function $f: \R \to [0, \infty)$, we can express $f$ by the layer-cake representation $f(x) = \int_0^\infty \ind_{\{y \in \R: f(y) > \ell\}} (x) d\ell$. We define its \emph{symmetric decreasing rearrangement} as 
\begin{equation*}
f^*(x) := \int_0^\infty \ind_{\{y \in \R: f(y) > \ell\}^*} (x) d\ell.
\end{equation*} 
\indentation It is straightforward to show that for arbitrary fixed $\ell \geq 0$, the sets $\{x: f(x) \geq \ell\}$ and $\{x: f^*(x) \geq \ell\}$ have the same Lebesgue measure. As a consequence, $\|f\|_{L^p(\R)} = \|f^*\|_{L^p(\R)}$ for every $p \geq 1$.
\bigskip
\\
Let $\rho:[0, T] \times \R \to [0, \infty)$. We define the \emph{Steiner symmetrization} of $\rho$ along the time-axis, denoted $\devs$, as follows: For every fixed $0 \leq s \leq T$, we define the function $\devs(s, \Cdot):= \rho^*(s, \Cdot)$.
\bigskip
\\
To prove Proposition \ref{prop:SD}, we rely on the following two rearrangement inequalities known respectively as the Hardy-Littlewood inequality and the Brascamp-Lieb-Luttinger inequality.
\begin{lem}[{\cite[Theorem 3.4]{LiebLoss}}]
	\label{lem:SDsingle}
	For any non-negative $f, g \in L^2(\R)$, one has 
	\begin{equation*}
	\int_\R f(x) g(x) dx \leq \int_\R f^*(x) g^*(x) dx.
	\end{equation*}
Moreover, if $f = f^*$ is strictly symmetric decreasing, then the inequality above becomes an equality if and only if $g = g^*$ almost everywhere.
\end{lem}
\begin{lem}[{\cite[Theorem 1.2]{brascamp1974general}}]
	\label{lem:SDmulti}
	Let $f_j, 1 \leq j \leq k$, be non-negative measurable functions on $\R$, and let $a_{jm}$, $1 \leq j \leq k$, $1 \leq m \leq n$ be real numbers, then 
	\begin{equation*}
	\int_{\R^n} \prod_{i=1}^k f_j(\sum_{m=1}^n a_{jm} x_m) d x_1 \dots dx_n \leq \int_{\R^n} \prod_{i=1}^k f_j^*(\sum_{m=1}^n a_{jm} x_m) d x_1 \dots dx_n.
	\end{equation*}
\end{lem}
\noindent The following result contains the major step for proving Proposition \ref{prop:SD}.
\begin{lem}\label{lem:ZfnSD}
	If $\dev \in L^2([0, 2] \times \R)$ is non-negative, then $\Zfn(\dev; 2, 0) \leq \Zfn(\devs; 2, 0)$. In addition, the equality holds if and only if $\dev$ is symmetric decreasing in space.	
\end{lem}
\begin{proof}
	Note that 
	\begin{equation}
	\label{e.l.chaoexpan}
	\Zfn(\dev; 2, 0) 
	= p(2, 0) + \sum_{n = 1}^\infty \int_{0 < t_n < \dots < t_1 < t_0} \int_{\mathbb{R}^n}  \prod_{i=1}^{n+1} p(t_{i-1}-t_i, x_{i-1} - x_i) \prod_{i=1}^n \dev(t_i, x_i) dt_i dx_i, 
	\end{equation}
	where $t_0 = 2$ and $x_0 = t_{n+1} = x_{n+1} = 0$. For the $n = 1$ term in the sum above, by Lemma \ref{lem:SDsingle}, 
	\begin{equation}\label{e.l.firstterm}
	\int_0^2 \int_\R p(2-s, x) p(s, x) \dev(s, x) dsdx \leq \int_0^2 \int_\R p(2-s, x) p(s, x) \devs (s, x) dsdx.
	\end{equation} 
	Since $p(2-s, \Cdot) p(s, \Cdot)$ is strictly symmetric decreasing for every fixed $s \in (0, 2)$, the above becomes an equality if and only if $\dev$ is symmetric decreasing in space. 
	\bigskip
	\\
	For the $n \geq 2$ term. View the integrand as product of functions in space for fixed $t_n <  t_{n-1} <  \dots < t_1$ and note that $p(t, \Cdot) = p^{\mathsf{s}}(t, \Cdot)$. By Lemma \ref{lem:SDmulti},
	\begin{align}\notag
	&\int_{0 < t_n < \dots < t_1 < t_0} \int_{\R^n} \prod_{i=1}^{n+1} p(t_{i-1} - t_i, x_{i-1} - x_i) \prod_{i=1}^n \dev(t_i, x_i) dt_i dx_i\\
	\label{e.l.higherterm}
	&\leq \int_{0 < t_n < \dots < t_1 < t_0} \int_{\R^n} \prod_{i=1}^{n+1} p(t_{i-1} - t_i, x_{i-1} - x_i) \prod_{i=1}^n \devs(t_i, x_i) dt_i dx_i.
	\end{align}
	Combining \eqref{e.l.chaoexpan}-\eqref{e.l.higherterm}, we conclude that $\Zfn(\dev; 2, 0) \leq \Zfn(\devs; 2, 0)$. Furthermore, the equality holds if and only if $\dev$ is symmetric decreasing in space.
\end{proof}
\begin{proof}[Proof of Proposition \ref{prop:SD}]
By the scaling \eqref{eq:calKscaling}, it suffices to prove Proposition \ref{prop:SD} for $\tcalK_\scl$. 
First, we claim that if $\dev \in \tcalK_\scl$, then $\dev \geq 0$ almost everywhere. 
Define $\dev_+ = \max(\dev, 0)$. The claim follows from $\Zfn(\dev; 2\scl, 0) \leq \Zfn(\dev_+; 2\scl, 0)$ and $\|\dev\|_{L^2([0, 2] \times \R)} > \|\dev_+\|_{L^2([0, 2] \times \R)}$, if $\rho$ is not almost everywhere non-negative. Next, suppose that there exists $\dev \in \tcalK_\scl$ that is not symmetric decreasing in space. By Lemma \ref{lem:ZfnSD}, we have $\Zfn(\devs; 2, 0) > \Zfn(\dev; 2, 0) \geq \frac{1}{\sqrt{4\pi}} e^\scl$. By the continuity of $\dev \mapsto \Zfn(\dev; 2, 0)$ in $L^2([0, 2] \times \R)$, there exists some $0 < \theta < 1$ such that $\Zfn(\theta \devs; 2, 0) = \frac{1}{\sqrt{4\pi}} e^{\scl}$. However,
\begin{equation*}
\|\theta \devs\|_{L^2([0, 2] \times \R)} < \|\devs\|_{L^2([0, 2] \times \R)} = \|\dev\|_{L^2([0, 2] \times \R)},
\end{equation*}
which contradicts with $\rho \in \tcalK_\scl$.
\end{proof}
\subsection{The $\e \to 0$ limit of $h_{\e, \scl}$ for fixed $\scl$}
\label{sec:elimit}
For $f \in C([\delta, 2] \times [-\delta^{-1}, \delta^{-1}])$ and $\mathcal{A} \subseteq C([\delta, 2] \times [-\delta^{-1}, \delta^{-1}])$, define $$\dist_\delta (f, \mathcal{A}) := \inf\{\|f - g\|_{L^\infty([\delta, 2] \times [-\delta^{-1}, \delta^{-1}])}: g\in \mathcal{A}\}.$$
Recall that we denote the infimum in \eqref{eq:tminimizer} by $\infi$.
\begin{lem}\label{lem:dist}
Consider $\{\rho_n\}_{n \in \Z_{\geq 1}} \subseteq L^2([0, 2] \times \R)$ such that $\Zfn(\rho_n; 2, 0) \geq 
\frac{1}{\sqrt{4\pi}} e^\scl$. If $\lim_{n \to \infty}\frac{1}{2}\|\rho_n\|_{L([0, 2] \times \R)}^2 = \infi$, then we have 
\begin{equation}\label{eq:miniconverg}
\lim_{n \to \infty}\dist_\delta (\Zfn(\rho_n), \Zfn(\tcalK_\scl)) = 0.
\end{equation}
\end{lem}
\begin{proof}
The proof 
is similar to the proof of Proposition \ref{prop:existence}.
Fix arbitrary $r \geq 0$. The set $\{\rho: \|\rho\|_{L^2([0, 2] \times \R)} \leq r\}$ is compact in $\bana$. 
Using this, after passing to a subsequence, $\dev_n$ converges to some $\rho \in L^2([0, 2] \times 
\R)$ in the topology $\bana$. In addition, 
$\Zfn(\dev; 2, 0) = \lim_{n \to \infty} \Zfn(\dev_n; 2, 0) \geq \frac{1}{\sqrt{4\pi}} e^\scl$ and 
$\|\dev\|_{L^2([0, 2] \times \R)} \leq \liminf_{n \to \infty} \|\dev_n\|_{L^2([0, 2] \times \R)}$.
Hence   
$\dev \in \tcalK_\scl$. This implies that 
\begin{equation}\label{eq:devconverg}
\lim_{n \to 
	\infty}\inf\{\|\dev_n - \dev\|_\bana: \rho \in \tcalK_\scl\} = 0.
\end{equation} 
Further, by Lemma \ref{lem:equicontinuity}, the map $\Zfn: L_\bana^2([0, 2] \times \R) \rev{\cap \{ \rho: \|\rho\|_{L^2([0, 2] \times \R)} \leq r \}}
\to C([\delta, 2] \times [-\delta^{-1}, \delta^{-1}])$ is continuous. 
Since $\{\rho: \|\rho\|_{L^2([0, 2] \times \R)} \leq r\}$ is compact in $\bana$, 
the map
$$\Zfn: \{\dev: \|\dev\|_{L^2([0, 2] \times \R)} \leq r\} 
\to C([\delta, 2] \times [-\delta^{-1}, \delta^{-1}])$$ is uniformly continuous, where we endow $ \{\dev:\|\dev\|_{L^2([0, 2] \times \R)} \leq r\}$ 
with the topology induced by $\bana$. Using the uniform continuity of $\Zfn$ and the convergence in  \eqref{eq:devconverg}, we conclude \eqref{eq:miniconverg}.
\end{proof}
\indentation Recall that $h_{\e, \scl} = \scl^{-1} \log  Z_\e (t, \scl^{\frac{1}{2}} x)$. Define $\hfn_\scl(\dev; t, x) := \scl^{-1}\log\big(\scl^{\frac{1}{2}} \Zfn (\dev; \scl t, \scl x)\big)$. 
\begin{prop}\label{prop:hfncalK}
Fix $\scl, \delta >0$. We have 
\begin{equation*}
\lim_{\e \to 0} \P\big[\dist_\delta (h_{\e, \scl}, \hfn_\scl(\calK_\scl)) < \delta\, \big|\, Z_\e (2, 0) \geq \frac{1}{\sqrt{4\pi}} e^\scl\big] = 1.
\end{equation*}
\end{prop}
\begin{proof}
Set $\hfnn_\scl (\rho; t, x) := \scl^{-1} \log \Zfn(\dev; t, \scl^{\frac{1}{2}} x)$. By the scaling \eqref{eq:calKscaling} and $\Zfn(\scl\dev(\scl \Cdot, \scl^{\frac{1}{2}} \Cdot); t, x) = \scl^{\frac{1}{2}} \Zfn(\dev; \scl t; \scl^{\frac{1}{2}} x)$, we have $\hfn_\scl (\calK_\scl) = \hfnn_\scl(\tcalK_\scl)$.  Hence, 
we need to prove 
\begin{equation*}
\lim_{\e \to 0} \P\big[\dist_\delta (\scl^{-1} \log  Z_\e (\Cdot, \scl^{\frac{1}{2}} \Cdot), \scl^{-1}\log\Zfn(\tcalK_\scl; \Cdot, \scl^{\frac{1}{2}} \Cdot)) \leq \delta\, \big|\, Z_\e (2, 0) \geq \frac{1}{\sqrt{4\pi}} e^\scl\big] = 1.
\end{equation*}
Since $\scl$ is fixed, we can drop the $\scl^{-1}$ in front of the $\log$ and $\scl^{\frac{1}{2}}$ in front of the $x$-variable. It suffices to prove 
\begin{equation}\label{eq:hfnlimit}
\lim_{\e \to 0}\P\big[\dist_\delta(\log Z_\e, \log \Zfn(\tcalK_\scl)) < \delta\, \big| \, Z_\e (2, 0) \geq \frac{1}{\sqrt{4\pi}} e^\scl\big] = 1.
\end{equation}
Let us first show that for fixed $\scl, \delta > 0$,
\begin{equation}\label{eq:condldp}
\lim_{\e \to 0}\P\big[\dist_\delta (Z_{\e}, \Zfn(\tcalK_\scl)) < \delta\, \big|\, Z_\e (2, 0)  \geq \frac{1}{\sqrt{4\pi}}e^\scl\big] = 1,
\end{equation}
and then explain how \eqref{eq:condldp} implies \eqref{eq:hfnlimit}.
\bigskip
\\
To prove \eqref{eq:condldp}, consider the set  $\gg_{\scl, \delta} := \{f \in C([\delta, 2] \times 
[-\delta^{-1}, \delta^{-1}]): f(2, 0) \geq \frac{1}{\sqrt{4\pi}} e^{\scl} \text{ and } \dist_\delta(f, \Zfn(\tcalK_\scl)) \geq \delta\}$. Since $\gg_{\scl, \delta}$ is a closed set in $C([\delta, 2] \times [-\delta^{-1}, \delta^{-1}])$, by Proposition \ref{prop:FWLDP}, we have 
\begin{equation}
\label{e.p.hfncalK1}
\limsup_{\e \to 0} \e \log \P\big[Z_\e \in \gg_{\scl, \delta}\big] \leq -\inf_{f \in \gg_{\scl, \delta}} I(f).
\end{equation}
By 
\eqref{eq:ptrate}, we also have 
\begin{equation}\label{e.p.hfncalK2}
\lim_{\e \to 0} \e \log \P\big[Z_\e \geq \frac{1}{\sqrt{4\pi}} e^\scl\big] = -q_\scl,
\end{equation}
recall that $\infi$ is equal to \eqref{eq:tminimizer}. Using $$\P\big[\dist_\delta(Z_\e, \Zfn(\tcalK_\scl)) \geq 
\delta \, \big|\, Z_\e(2, 0) \geq \frac{1}{\sqrt{4\pi}} e^\scl\big] = \frac{\P[Z_\e \in \gg_{\scl, \delta}]}{\P[Z_\e(2, 0) \geq \frac{1}{\sqrt{4\pi}} e^{\scl}]}$$ and applying \eqref{e.p.hfncalK1}-\eqref{e.p.hfncalK2}, we have
\begin{equation*}
\limsup_{\e \to 0} \e \log \P\big[\dist_\delta(Z_\e, \Zfn(\tcalK_\scl)) \geq 
\delta \, \big|\, Z_\e(2, 0) \geq \frac{1}{\sqrt{4\pi}} e^\scl\big] \leq q_\scl - \inf_{f \in \gg_{\scl, \delta}} I(f).
\end{equation*}
By Lemma \ref{lem:dist}, there exists $\z > 0$ such that $\inf_{f \in \gg_{\scl, \delta}} I(f) \geq q_\scl + \z$. The above inequality implies $$\lim_{\e \to 0}\P\big[\dist_\delta (Z_{\e}, \Zfn(\tcalK_\scl)) \geq \delta\, \big|\, Z_\e (2, 0)  \geq \frac{1}{\sqrt{4\pi}}e^\scl\big] = 0.$$
Therefore, We conclude \eqref{eq:condldp}.
\bigskip
\\
Finally, we show how \eqref{eq:condldp} implies \eqref{eq:hfnlimit}. 
Recall that $\Zfn(\dev;t, x) = P(\dev; (0, 0) \to (t, x))$. Referring to \eqref{eq:semigroup}, due to the non-negativity of $\dev$, we have the lower bound $\Zfn(\dev; t, x) \geq p(t, x)$. The upper bound of $\Zfn(\dev; t, x)$ is given by \eqref{e.l.ptpt}. This implies the existence of a constant $M$ such that for all $\rho \in \tcalK_\scl$, 
\begin{equation}\label{eq:minimizerbd}
M^{-1} \leq \|\Zfn(\dev)\|_{L^\infty([\delta, 2] \times [-\delta^{-1}, \delta^{-1}])} \leq M.
\end{equation}  
\eqref{eq:hfnlimit} follows from \eqref{eq:condldp}, \eqref{eq:minimizerbd} and the uniform continuity of the $\log$ function on the interval $[(2M)^{-1}, 2M]$. 
\end{proof}

\subsection{The $\scl \to \infty$ asymptotic of $\calK_\scl$}
\label{sec:lambdainfty}
\indentation For simplicity, we adopt the shorthand notation $\normLL{\Cdot} := \normL{\Cdot}$. The functional $F$ defined in \eqref{eq:funcF}
enjoys a scaling property: Letting $\varphi_\alpha (x) = \alpha^2 \varphi(\alpha x)$, we have $F(\varphi_\alpha) = \alpha^2  F(\varphi)$. Choosing $\alpha = \|\varphi\|_2^{-2/3}$ to normalize the $L^2$ norm of $\varphi$, we have 
\begin{equation}\label{eq:scaling}
F(\varphi) = \|\varphi\|_{2}^{\frac{4}{3}} F(\varphi_{\normLL{\varphi}^{-2/3}}), \qquad \|\varphi_{\normLL{\varphi}^{-2/3}}\|_2 = 1.
\end{equation}
\noindent 
We set 
\begin{equation}\label{eq:rstar}
\rr_\star(x) := (\frac{3}{4})^{2/3} \sech^2((\frac{3}{4})^{1/3} x). 
\end{equation}
Also, let $\SD$ denote 
the space of symmetric decreasing functions (see Section \ref{sec:sdf} for the definition). 
\begin{lem}\label{lem:reversebd}
For any $\z > 0$, there exists $\delta > 0$ such that for all $\pot \in \SD$,
\begin{equation*}
F(\pot) \geq \frac{1}{2}(\frac{3}{4})^{\frac{2}{3}} \normLL{\pot}^{\frac{4}{3}} (1-\delta) \ \text{ implies }\  \|\pot_{\normLL{\pot}^{-2/3}} - \rr_\star\|_2 < \z.
\end{equation*}
\end{lem}   
\begin{proof}
By the scaling \eqref{eq:scaling} we assume without loss of generality that $\|\varphi\|_2 = 1$. The desired statement is equivalent to the following: Any sequence $\{\varphi_n\} \subseteq \SD$ with $\|\varphi_n\|_2 = 1$ such that $F(\varphi_n) \to \frac{1}{2}(\frac{3}{4})^{\frac{2}{3}}$ has a subsequence that converges in $L^2$ to $\rr_\star$. Fix any such $\varphi_n$. Let $f_n \in H^1(\R)$ be such that $\|f_n\|_{2} = 1$ and $\int \varphi_n f_n'^2 - \frac{1}{2} f_n'^2 dx \to \frac{1}{2}(\frac{3}{4})^{2/3}$. Since $\{(\varphi_n, f_n, f_n')\}$ is bounded in $(L^2(\R))^3$, by the Banach-Alaoglu theorem, after passing to a subsequence, $(\varphi_n, f_n, f_n')$ converges weakly to $(\rr_*, f_*, g_*)$ with $\|r_*\|_2 \leq 1$, $\|f_*\|_2 \leq 1$ and $\|g_*\| \leq \liminf_{n \to \infty}\|f_n'\|_2$. Since $(f_n, f_n')$ converges weakly to $(f_*, g_*)$, we must have $f'_* = g_*$. Moreover, since $f_n$ and $f_n'$ are bounded in $L^2(\R)$, $f_n$ is locally uniformly bounded and equi-continuous. By the Arzela-Ascoli theorem, we know that $f_n \to f_*$ on compact sets. Next, because $\|\varphi_n\|_2 = 1$ and because $\varphi_n \in \SD$, $\varphi_n(\pm M) \leq M^{-\frac{1}{2}}$. Therefore, $\int_{|x| > M} \varphi_n f_n^2 dx \leq M^{-\frac{1}{2}} \|f_n\|_2^2 = M^{-\frac{1}{2}}$. This gives 
\begin{equation*}
\limsup_{n \to \infty} \int_\R \varphi_n f_n^2 \ind_{\{|x| \leq M\}} - \frac{1}{2} f_n'^2 dx \geq \frac{1}{2} \Big(\frac{3}{4}\Big)^{\frac{2}{3}} - M^{-\frac{1}{2}}.
\end{equation*}
The left hand side is bounded above by $\int_\R  \rr_* f_*^2 -\frac{1}{2} {f_*'}^2 dx$. Sending $M \to \infty$ gives $\int_\R \rr_* f_*^2  - \frac{1}{2} {f'_*}^2 dx \geq \frac{1}{2}(\frac{3}{4})^{2/3}$. This together with $\|\rr_*\|_2 \leq 1$ and $\|f_*\|_2 \leq 1$ implies that $F(\rr_*) \geq \frac{1}{2}(\frac{3}{4})^{2/3}$. By the if and only if part in Lemma \ref{lem:potbd} and the fact that $\rr_* \in \SD$, we know that $\rr_* = \rr_\star$. Since $\varphi_n$ weakly converges to $\rr_\star$ and $\lim_{n \to \infty} \|\varphi_n\|_2 = 1 = \|\rr_\star\|_2$, we conclude that $\varphi_n$ converges to $\rr_\star$ in $L^2$.
\end{proof}
\indentation Recall that $\devm(t, x) = \devm (x) = (\sech x)^2$. The main result of this section says that the elements of $\calK_\scl$ tend to $\rho_*$ as $\scl \to \infty$.
\begin{prop}\label{prop:devmdist}
	We have $\lim_{\scl \to \infty}\sup\{\frac{1}{2\scl}\|\dev - \devm\|^2_{L^2([0, 2\scl] \times \R)}: \rho 
	\in \calK_\scl\} = 0$.
\end{prop}
	\begin{proof}
Set $\z > 0$. According to \eqref{eq:heuristic1}, which is proved in Section \ref{sec:uptail}, for large enough $\scl$ and all $\rho \in \calK_\scl$, 
\begin{equation}\label{eq:l2upbd}
\frac{1}{2\scl}\|\rho\|^2_{L^2([0, 2\scl] \times \R)} \leq \frac{4}{3} + \z.
\end{equation} 
By $\Zfn(\rho; 2\scl, 0) \geq \frac{1}{\sqrt{4\pi \scl}} e^\scl$ and \eqref{e.t.uptail},
\begin{align*}
\frac{1}{\sqrt{4\pi \scl}}e^\scl 
\leq 
C\exp\Big(C \scl^{2/3} + \int_0^{2\scl} F(\dev(r, \Cdot)) dr\Big).
\end{align*}
Taking the logarithm of both sides and then applying Lemma \ref{lem:potbd} to the right hand side yields that for all $\scl$ large enough, 
\begin{equation*}
\scl(1-\z) \leq \int_0^{2\scl} F(\dev(r, \Cdot)) dr \leq \int_0^{2\scl }\frac{1}{2}\Big(\frac34\Big)^{\frac23} \|\dev(r, \Cdot)\|_2^{\frac43} dr.
\end{equation*}
Young's inequality gives ($(\frac{3}{4} \|\rho(r, \Cdot)\|_2^2)^{2/3} \leq \frac{1}{3} + 	\frac{2}{3}(\frac{3}{4} \|\dev(r, \Cdot)\|_2^2)$). Furthermore, there exists a strictly increasing function $\psi: \R_{\geq 0} \to \R$ such that $\psi(0) = 0$, $(\frac{3}{4} \|\rho(r, \Cdot)\|_2^2)^{2/3} \leq \frac{1}{3} +  \frac{2}{3}(\frac{3}{4} \|\dev(r, \Cdot)\|_2^2) - \psi(|1 - \frac{3}{4}\|\rho(r, \Cdot)\|_2^2|)$. Applying this to the right hand side above yields
\begin{align*}
\lambda(1-\z) \leq \int_0^{2\scl} F(\rho(r, \Cdot)) dr
\leq \int_0^{2\scl }\frac{1}{2}\Big(\frac34\Big)^{\frac23} \|\dev(r, \Cdot)\|_2^{\frac43} dr 
&\leq \int_0^{2\scl} \frac{1}{6} + \frac{1}{4} \|\rho(r, \Cdot)\|_2^2 - \frac{1}{2} \psi(1- \frac{3}{4}\|\rho(r, \Cdot)\|_2^2) dr\\
&\leq \lambda(1+\z) - \int_{0}^{2\scl} \frac{1}{2}\psi(1- \frac{3}{4} \|\rho(r, \Cdot)\|_2^2) dr.
\end{align*}
The last inequality is due to \eqref{eq:l2upbd}. 
Using the inequalities above, we have for all $\scl$ large enough,
\begin{align}\label{eq:distbd1}
\int_0^{2\scl} \psi(|1 - \frac{3}{4} \|\rho(r, \Cdot)\|_2^2|) dr \leq 4\scl \z, 
\\
\label{eq:distbd2}
\int_0^{2\lambda} \frac{1}{2} \Big(\frac{3}{4}\Big)^{\frac{2}{3}} \|\rho(r, \Cdot)\|_2^{\frac{4}{3}} - F(\rho(r, \Cdot)) dr \leq 4\scl \z.
\end{align}
Since $\psi$ is strictly increasing with $\psi(0) = 0$. 	For any $\zeta' > 0$, we have $\psi(z) \geq \psi(\z') =: C(\z')$ when $z \geq \z'$. Applying this to \eqref{eq:distbd1} gives that for all $\scl$ large enough, 
\begin{equation}\label{eq:l2normdiff}
\text{Leb}\bigg[\Big\{r \in [0, 2\scl]: \Big|1 - \frac{3}{4}\|\rho(r, \Cdot)\|_2^2\Big| > \z'\Big\}  \bigg] \leq \frac{4 \scl \z}{C(\z')}.
\end{equation}
Here, $\text{Leb}$ denotes the Lebesgue measure.
Since $\rho \in \calK_\scl$, by Proposition \ref{prop:SD}, $\rho(r, \Cdot) \in \SD$ for almost all $r \in [0, 2\scl]$. Using \eqref{eq:distbd2} and Lemma \ref{lem:reversebd}  we conclude that for any $\z' > 0$, there exists $C(\zeta') > 0$ such that 
\begin{equation}\label{eq:devmdistpp}
\Leb\bigg[\Big\{r \in [0, 2\scl]: \|\rho(r, \Cdot)_{\|\rho(r, \Cdot)\|_2^{-2/3}} - r_\star\|_2 > \z'\Big\}\bigg] \leq \frac{\scl\z}{C(\z')}, 
\end{equation}
where we denote $\rho(r, \Cdot)_\alpha = \alpha^2 \rho(r, \alpha \Cdot)$.  We have
\begin{equation*}
\|\rho(r, \Cdot)_{\|\rho(r, \Cdot)\|_2^{-2/3}} - r_\star\|_2 = \|\rho(r, \Cdot)\|_2^{-1} \|\rho(r, \Cdot) - (r_\star)_{\|\rho(r, \Cdot)\|_2^{2/3}}\|_2.
\end{equation*} 
Applying this to \eqref{eq:devmdistpp} yields
\begin{equation}\label{eq:devmdistp}
\Leb\bigg[\Big\{r \in [0, 2\scl]: \|\rho(r, \Cdot) - (r_\star)_{\|\rho(r, \Cdot)\|_2^{2/3}}\|_2 > \z' \|\rho(r, \Cdot)\|_2 \Big\}\bigg] \leq \frac{\scl\z}{C(\z')}. 
\end{equation}
Taking $\z \leq \min(C(\zeta'), 1) \z'$, the right hand side of \eqref{eq:devmdistp} is upper bounded by $\scl \z'$. By \eqref{eq:l2normdiff}, we know that for most $r \in [0, 2\scl]$, $\|\dev(r, \Cdot)\|_2$ is around $\sqrt{\frac{4}{3}}$. Applying this and $(r_\star)_{(\sqrt{\frac{4}{3}})^{2/3}} = \rho_*$ to \eqref{eq:devmdistp} gives that for any $\z' > 0$,
\begin{equation}\label{eq:devmdist}
\Leb\bigg[\Big\{r \in [0, 2\scl]: \|\rho(r, \Cdot) - \rho_*\|_2 > \z'\Big\}\bigg] \leq C \scl\z'
\end{equation}
holds for all $\scl$ large enough.
The inequality \eqref{eq:devmdist} together with the fact $\|\devm\|_2^2  =\frac{4}{3}$ 
implies 
\begin{equation*}
\frac{1}{2\scl}\int_0^{2\scl} \|\rho(r, \Cdot)\|_2^2 \ind_{\{\|\rho(r, \Cdot) - \devm\|_2 \leq \z'\}} dr \geq \frac{4}{3} - C\z'.
\end{equation*}
This, together with \eqref{eq:l2upbd} (we have taken $\z \leq \z'$)	
imply that 
\begin{equation}\label{eq:devld}
\frac{1}{2\scl} \int_0^{2\scl} \|\dev(r, \Cdot)\|_2^2 \ind_{\{\|\rho(r, \Cdot) - \devm\|_2 > \z'\}} dr \leq C\z'.
\end{equation}
In addition, by $\|\devm(r, \Cdot)\|_2^2  =\frac{4}{3}$ and \eqref{eq:devmdist}, 
\begin{equation}\label{eq:devmld}
\frac{1}{2\lambda} \int_0^{2\scl} \|\devm(r, \Cdot)\|_2^2 \ind_{\{\|\rho(r, \Cdot) - \devm\|_2 > \z'\}} dr \leq C \z'. 
\end{equation}
Combining \eqref{eq:devld} and \eqref{eq:devmld}, we conclude that 
\begin{align*}
\frac{1}{2\scl} \|\dev - \devm\|_{L^2([0, 2\scl] \times \R)}^2  &= \frac{1}{2\scl} \int_0^{2\scl} \|\dev(r, \Cdot) - \devm(r, \Cdot)\|^2_2 \ind_{\{\|\dev(r, \Cdot) - \devm\|_2 > \z'\}} dr
\\ 
&\quad + \frac{1}{2\scl} \int_0^{2\scl}\|\dev(r, \Cdot) - \devm(r, \Cdot)\|^2_2 \ind_{\{\|\dev(r, \Cdot) - \devm\|_2 \leq \z'\}} dr
\leq C\z'.
\end{align*}
Letting $\scl \to \infty$ and then $\z' \to 0$ concludes the proposition.
\end{proof}
 
\section{The Limit shape}\label{sec:limitshape}
\indentation The goal of the section is to prove Theorem \ref{thm:limshape}.
\subsection{Equi-continuity of $\hfn_\scl$}
\indentation Recall from Section \ref{sec:elimit} that $\hfn_\scl(\dev; t, x) = \scl^{-1}\log\big(\scl^{\frac{1}{2}} \Zfn (\dev; \scl t, \scl x)\big)$.
Let $L_{\geq 0}^2([0, 2\scl] \times \R)$ denote the set of non-negative functions which belong to $L^2([0, 2\scl] \times \R)$. Consider the normalized norm 
$\scl^{-\frac{1}{2}}\|\Cdot\|_{L^2([0, 2\scl] \times \R)}$. The following proposition settles the equi-continuity of $\hfn_\scl(\Cdot, t, x): L_{\geq 0}^2([0, 2\scl] \times \R) \to \R$ with respect to this norm.
\begin{prop}\label{prop:uniformdist}
There exists a constant $C$ such that for all $\scl > 0, \dev_1, \dev_2 \in L_{\geq 0}^2([0, 2\scl] \times \R)$ and $(t, x) \in (0, 2] \times \R$, if $\lambda^{-\frac{1}{2}}\|\rho_1 - \rho_2\|_{L^2([0, 2\scl] \times \R)} < 1$, 
\begin{equation*}
|\hfn_\scl (\dev_1; t, x) - \hfn_\scl (\dev_2; t, x)| \leq C \scl^{-\frac{1}{2}}\|\rho_1 - \rho_2\|_{L^2([0, 2\scl] \times \R)} (1+ \scl^{-1}\|\rho_1\|_{L^2([0, 2\scl] \times \R)}^2 + \scl^{-1}\|\rho_2\|_{L^2([0, 2\scl] \times \R)}^2).
\end{equation*}
\end{prop} 
\begin{proof}
Set $\Zfnn(\dev; t, x) = \Zfn(\dev; t, x) / p(t, x)$. Since 
\begin{align*}
\scl^{-1} \log \Zfnn(\rho_1; \scl t, \scl x) - \scl^{-1} \log \Zfnn(\rho_2; \scl t, \scl x) &= \scl^{-1} \log \Zfn(\rho_1; \scl t, \scl x) - \scl^{-1} \log \Zfn(\rho_2; \scl t, \scl x)\\ 
&= \hfn_\scl (\rho_1; \scl t, \scl x) - \hfn_\scl (\rho_2; \scl t, \scl x),
\end{align*} 
it suffices to prove that for $\scl > 0$ and $(t, x) \in (0, 2] \times \R$,
\begin{align}\notag
&\scl^{-1} \big| \log \Zfnn(\rho_1; \scl t, \scl x)
- \log \Zfnn(\rho_2; \scl t, \scl x)\big|\\
\label{eq:desiredbd} 
&\leq C \scl^{-\frac{1}{2}}\|\rho_1 - \rho_2\|_{L^2([0, 2\scl] \times \R)} (1+ \scl^{-1}\|\rho_1\|_{L^2([0, 2\scl] \times \R)}^2 + \scl^{-1}\|\rho_2\|_{L^2([0, 2\scl] \times \R)}^2).
\end{align}
Note that $\Zfn(\dev; t, x) = P(\dev; (0, 0) \to (t, x))$. We apply \eqref{e.l.ptpt} with $a = 1$ to obtain 
\begin{equation}\label{eq:ptbd}
\Zfnn(\dev; \scl t, \scl x) \leq C \exp(C \scl t + \|\dev\|^2_{L^2([0, \scl t] \times \R)}).
\end{equation} 
The constant $C$ here is universal. Further, by the Feynman-Kac formula, 
\begin{equation*}
\Zfnn(\dev; \scl t, \scl x) = \E_{0 \to \scl x}\Big[\exp\Big(\int_0^{\scl t} \dev(s, \bb(s))\Big)\Big]. 
\end{equation*}
Let us use \eqref{eq:ptbd} and the Feynman-Kac formula  to prove \eqref{eq:desiredbd}.
Fix $\z \in (0, 1)$. Similar to the proof of Lemma \ref{lem:continuity}, by the Feynman-Kac formula and the H\"{o}lder inequality,
\begin{align*}
\Zfnn(\rho_1; \scl t, \scl x) 
&\leq 
\Zfnn(\z^{-1} (\dev_1 - (1-\z) \dev_2); \scl t, \scl x)^{\z} \Zfnn(\dev_2; \scl t, \scl x)^{1-\z}.
\end{align*}
Taking the logarithm of both sides  yields 
\begin{align}
\notag
\log \Zfnn(\rho_1; \scl t, \scl x)
&\leq \z \log \Zfnn(\z^{-1} (\rho_1 - (1-\z) \rho_2); \scl t, \scl x) + (1-\z) \log \Zfnn(\rho_2; \scl t, \scl x)
\\
\label{e.p.uniformdist}
&\leq \z \log \Zfnn(\z^{-1} (\rho_1 - (1-\z) \rho_2); \scl t, \scl x) + \log \Zfnn(\rho_2; \scl t, \scl x).
\end{align}
The last inequality is due to $\Zfnn(\dev_2; t, x) \geq 1$, since $\rho_2$ is non-negative.
\bigskip
\\
Subtracting $\log \Zfnn(\rho_2; \scl t, \scl x)$ from both sides of \eqref{e.p.uniformdist} and  
applying \eqref{eq:ptbd} to the resulting right hand side, we get
\begin{align}
\notag
\log \Zfnn(\rho_1; \scl t, \scl x) - \log \Zfnn(\rho_2; \scl t, \scl x) 
\notag
&\leq \z \log \Zfnn(\z^{-1} (\rho_1 - (1-\z) \rho_2); \scl t, \scl x) 
\\
\notag
&\leq C\z \scl t +  C \z^{-1} \|\rho_1 - (1-\z) \rho_2\|^2_{L^2([0, \scl t] \times \R)}.
\end{align}
Applying $t \leq 2$ and the inequality $\|f + g\|^2_{L^2} \leq 2(\|f\|_{L^2}^2 + \|g\|_{L^2}^2)$ to the right hand side leads to
\begin{align}
\label{e.p.uniformdist1}
\log \Zfnn(\rho_1; \scl t, \scl x) - \log \Zfnn(\rho_2; \scl t, \scl x) \leq C \z \scl + C \z^{-1} \|\dev_1 - \dev_2\|^2_{L^2([0, \scl t] \times \R)} + C\z \|\dev_2\|^2_{L^2([0, \scl t] \times \R)}.
\end{align}
The constant $C$ is universal and does not depend on $\scl, t, x$. 
Swapping $\dev_1, \dev_2$ in \eqref{e.p.uniformdist1} gives that 
\begin{equation*}
\big|\log \Zfnn(\rho_1; \scl t, \scl x) - \log \Zfnn(\rho_2; \scl t, \scl x)\big| \leq C\Big(\z^{-1} \|\rho_1 - \rho_2\|^2_{L^2([0, \scl t] \times \R)} + \z(\scl + \|\rho_1\|^2_{L^2([0, \scl t] \times \R)} + \|\rho_2\|^2_{L^2([0, \scl t] \times \R)}) \Big). 
\end{equation*}
Dividing both sides by $\scl$ and using $t \leq 2$, we have 
\begin{align}\notag
&\scl^{-1} \big| \log \Zfnn(\rho_1; \scl t, \scl x)
- \log \Zfnn(\rho_2; \scl t, \scl x)\big|\\
&\leq  C\Big(\z^{-1} \scl^{-1} \|\rho_1 - \rho_2\|^2_{L^2([0, 2\scl] \times \R)} + \z \scl^{-1} (\scl + \|\rho_1\|^2_{L^2([0, 2\scl] \times \R)} + \|\rho_2\|^2_{L^2([0, 2\scl] \times \R)}) \Big). 
\end{align}
Taking $\z = \scl^{-\frac{1}{2}} \|\rho_1 -\rho_2\|_{L^2([0, 2 \scl] \times \R)} \in (0, 1)$, 
we conclude \eqref{eq:desiredbd}.
\end{proof}
\subsection{Proof of Theorem~\ref{thm:limshape}}
We begin with a reduction.
Recall that Proposition~\ref{prop:hfncalK} states, for arbitrary fixed $\scl, \delta > 0$,  
\begin{equation*}
\lim_{\e \to 0} \P\big[\dist_\delta (h_{\e, \scl}, \hfn_\scl(\calK_\scl)) < \delta\, \big|\, h_\e (2, 0) + \log \sqrt{4\pi} \geq \scl\big] = 1.
\end{equation*} 
Combining Propositions \ref{prop:devmdist} and \ref{prop:uniformdist} gives 
$\lim_{\scl \to \infty}\dist_\delta (\hfn_\scl(\calK_\scl), \hfn_\scl (\rho_*)) = 0$.
Given these results, it suffices to prove
\begin{equation}\label{eq:hfntoh*}
\lim_{\scl \to \infty}\dist_\delta (\hfn_\scl (\rho_*),\hfn_*) = 0,
\qquad
\text{for any } \delta > 0.
\end{equation}

\medskip \indentation 
The proof of \eqref{eq:hfntoh*} starts with the Feynman-Kac formula.
Set $ \varphi = \rho_* $ in \eqref{eq:FK} and take logarithm on both sides to get
\begin{align}
\label{e.pf.limitshape.1}
\hfn_\scl(\rho_*; t, x) 
&= \scl^{-1} \log \E_{\scl x \to 0}\Big[\exp\Big(\int_0^{\scl t} \devm(\bb(s)) ds\Big)\Big] - \frac{x^2}{2t} - \scl^{-1} \log \sqrt{4\pi}.
\end{align}
For $ U_i=U_i(t,x,\lambda)>0 $, we write $ U_1 \sim U_2 $ if $ \lim_{\scl\to\infty} \scl^{-1}\log (U_1/U_2) =0 $ uniformly over $ (t,x)\in[\delta,2]\times[-\delta^{-1},\delta^{-1}] $.
Our goal is to estimate the expectation in \eqref{e.pf.limitshape.1}.
Fix a mesoscopic scale $ \lambda^a $. Any $ a\in(0,\frac12) $ will do, and we fix $ a=\frac14 $ for the sake of concreteness.
The Brownian bridge in \eqref{e.pf.limitshape.1} starts at $ \bb(0)=\lambda x $ and returns to $ \bb(\lambda t) = 0 $.
Consider the first time the bridge enters the region $ [-\lambda^\frac14,\lambda^\frac14] $, namely $ \tau := \inf\{ s \geq 0 : |\bb(s)| \leq \lambda^\frac14 \} $.
Decompose the integral in \eqref{e.pf.limitshape.1} into $ \int_0^{\tau} + \int_{\tau}^{\lambda t} $.
In the first integral, we have $ |\bb(s)| > \lambda^\frac14 $, which gives $ |\devm(\bb(s))| \leq \exp(-\lambda^{1/4}/C) $.
This shows that the contribution of $ \devm(\bb(s)) $ within $ s\in[0,\tau] $ is negligible, so
\begin{align}
\label{e.pf.limitshape.2}
\E_{\scl x \to 0}\Big[\exp\Big(\int_0^{\scl t} \devm(\bb(s)) \d s\Big)\Big] 
\sim
\E_{\scl x \to 0}\Big[\exp\Big(\int_{\tau}^{\scl t} \devm(\bb(s)) \d s\Big)\Big]. 
\end{align}
To proceed, condition on $ \tau $ in \eqref{e.pf.limitshape.2}, namely $ \E_{\scl x \to 0}[\Cdot] = \E_{\scl x \to 0}[\E[\Cdot|\tau]] $,
and recognize the inner expectation as $ \E[\exp(\int_{\tau}^{\scl t} \devm(\bb(s)) \d s)|\tau] = U(\tau,t) $, where $ U(s',t):=\E_{\lambda^{1/4}\to 0}[\exp(\int_{0}^{\scl t-s'} \devm(\bb(s))\d s )] $.
Fix any $ \zeta\in(0,\delta) $.
By Lemma \ref{lem:bboptime} and \ref{lem:potbd}, 
$ U(s',t) \sim \frac12 \exp(\lambda t-s') $, uniformly over $ t \in [\delta,2] $ and $ s'\in[0,\lambda t-\lambda \zeta] $.
This approximation extends to $ s'\in[0,\lambda t] $.
To see why, note that when $ s'\in(\lambda t-\lambda\zeta,\lambda t] $, we have $ |\int_{\tau}^{\scl t} \devm(\bb(s)) \d s| \leq \lambda \zeta $, 
and note that $ \zeta $ can be taken to be arbitrarily small.
Therefore,
\begin{align}
\label{e.pf.limitshape.3}
\E_{\scl x \to 0}\Big[\exp\Big(\int_0^{\scl t} \devm(\bb(s)) \d s\Big)\Big] 
\sim
e^{\frac{\scl t}{2}} \E_{\scl x \to 0}[ e^{-\frac{1}{2}\tau} ].
\end{align}

\medskip \indentation 
Next we bound the right side of \eqref{e.pf.limitshape.3}.
By symmetry it suffices to consider $ x \geq 0 $, which we assume hereafter.
Let $ \mathsf{T}(s,u) $ be the first hitting time of $ 0 $ of the Brownian bridge that starts from $ u $ at time $ 0 $ and returns to $ 0 $ at time $ s $.
For fixed $u >0 $, $ \mathsf{T}(s, u) $ is stochastically increasing in $ s $.
This fact can be proven by expressing the bridge as a drifted Brownian motion via Doob{}'s h transform.
Recall that $ \tau $ is the first time $ \bb $ enters $ [-\scl^{1/4},\scl^{1/4}] $.
After the first entrance, consider the excess amount of time it takes for $ \bb $ to hit $ 0 $, namely $ \sigma := \inf\{ s \geq 0 : \bb(s+\tau) = 0 \} $.
Indeed, $ \tau+\sigma = \mathsf{T}(\scl t, \scl x) $, so $ \E_{\scl x \to 0}[ e^{-\frac{1}{2}\tau} ] = \E[ e^{-\frac{1}{2}\mathsf{T}(\scl t, \scl x)} e^{\frac12 \sigma} ] $.
Conditioned on $ \tau $, $ \sigma $ is equal in law to $ \mathsf{T}(\scl t-\tau, \scl^{1/4}) $, which is stochastically bounded above by $ \mathsf{T}(\scl t, \scl^{1/4}) $.
Using H\"{o}lder's inequality and the stochastic bound, we have
\begin{align}
\label{e.pf.limitshape.4}
\E[ e^{-\frac{1}{2}\mathsf{T}(\scl t, \scl x)} ]
\leq
\E_{\scl x \to 0}[ e^{-\frac{1}{2}\tau} ] 
\leq 
\big( \E[ e^{-\frac{n+1}{2n}\mathsf{T}(\scl t, \scl x)} ] \big)^{\frac{n}{n+1}}
\big( \E[ e^{\frac{n+1}{2}\mathsf{T}(\scl t, \scl^{1/4})} ] \big)^{\frac{1}{n+1}}.
\end{align}

\medskip \indentation 
Given \eqref{e.pf.limitshape.4}, we seek to estimate $ \E[ \exp(-\beta\mathsf{T}(\scl t, \scl x)) ] $.
The first step is to derive the probability density function of $ \mathsf{T}(\scl t, \scl x) $.
Express the Brownian bridge $ \bb $ by a Brownian motion as $ \bb(s) =  (1-\frac{s}{\scl t})(\scl x - \sqrt{\scl t}  B(\frac{s}{\scl t-s})) $;
relate $ \mathsf{T}(\scl t, \scl x) $ to a hitting time of $ B $;
use the known density function of the hitting time of $ B $.
The result reads
$
(\text{density function of }\mathsf{T}(\scl t, \scl x))(s)
=
(\sqrt{\scl^3 t x^2}/\sqrt{2\pi s^3(\scl t -s)})
\exp( -\tfrac{\scl t - s}{2\scl t s} (\scl x)^2 ).
$
Use this density function to express $ \E[ \exp(-\beta\mathsf{T}(\scl t, \scl x)) ] $  as an integral, and perform a change of variables $ s \mapsto \scl t s $. We have
\begin{align}
\label{e.pf.limitshape.5}
\E[ e^{-\beta\mathsf{T}(\scl t, \scl x)} ] 
=
\int_0^{1} \frac{\sqrt{\scl x^2}}{\sqrt{2\pi t s^3(1-s)}} \exp\big( \scl \, V_\beta(s,t,x) \big) \d s,
\end{align}
where $ V_\beta(s,t,x) := -\beta t s - \frac{1-s}{2st}x^2 $.
This integral can be analyzed by Laplace's method.
Differentiating $ V $ in $ s $, one finds that $ V(\Cdot,x) $ attains its unique maximum at $ s=\min\{ \frac{x}{\sqrt{2\beta} \, t}, 1 \} $, and $ \partial^2_{s} V_\beta = \frac{x^2}{t s^3} $.
Using these properties in \eqref{e.pf.limitshape.5}, it is not hard to show that
$ \E[ \exp(-\beta\mathsf{T}(\scl t, \scl x)) ] \sim \exp(-\scl t V_\beta(\min\{ \frac{x}{\sqrt{2\beta} \, t}, 1 \}, t ,x)) $,
for fixed $ \beta>0 $ and uniformly over $ [t,x]\in[\delta,2]\times[0,\delta^{-1}] $.
Insert this estimate into \eqref{e.pf.limitshape.4}; take $ \scl^{-1}\log(\Cdot) $ of the result; send $ \scl\to\infty $ first and $ n\to\infty $ later.
We obtain
\begin{align*}
\lim_{\scl\to\infty} \scl^{-1} \log
\E_{\scl x \to 0}[ e^{-\frac{1}{2}\tau} ] 
=
V_{1/2}(\min\{\tfrac{x}{t},  1\}, t, x) 
= 
\left\{\begin{array}{l@{,}l}
\frac{x^2}{2t} - x	& \text{ when } x\in[0,t],
\\
-\frac{t}{2} & \text{ when } x>t,
\end{array}\right.
\end{align*}
uniformly over $ [t,x]\in[\delta,2]\times[0,\delta^{-1}] $.
Inserting this into \eqref{e.pf.limitshape.3} and then inserting the result into \eqref{e.pf.limitshape.1} completes the proof of \eqref{eq:hfntoh*} and hence the proof of Theorem~\ref{thm:limshape}.
\bibliographystyle{alpha}                         
\bibliography{ref.bib}                            
\end{document}